\documentclass[a4paper,11pt]{article}
\usepackage[T2A]{fontenc}    
\usepackage[utf8]{inputenc} 

\usepackage[english,russian,main=english]{babel}
\usepackage{csquotes}
\usepackage[backend=bibtex,style=numeric ,natbib=true, maxnames=10, backref=true]{biblatex}
\usepackage{booktabs} 

\usepackage{url}            
\usepackage{booktabs}       
\usepackage{amsfonts}       
\usepackage{nicefrac}       
\usepackage{microtype}      
\usepackage{xcolor}         
\usepackage{amsmath}
\usepackage{graphicx}
\graphicspath{{./figures/}}
\usepackage{caption}
\usepackage{amsthm}
\usepackage{amssymb}
\usepackage{hyperref}       
\usepackage{multirow}

\newtheorem{theorem}{Theorem}[section]

\title{Another approach to build Lyapunov functions for the first order methods in the quadratic case}

\author{%
  Daniil~Merkulov\thanks{
  Skolkovo Institute of Science and Technology} \thanks{Moscow Institute of Physics and Technology}\\
  \texttt{d.merkulov@skoltech.ru}
  \and
  Ivan~Oseledets\thanks{
  Artificial Intelligence Research Institute} \footnotemark[1]\\
  \texttt{i.oseledets@skoltech.ru}
}

\date{}

\begin{document}

\maketitle

\begin{abstract}
    Lyapunov functions play a fundamental role in analyzing the stability and convergence properties of optimization methods. In this paper, we propose a novel and straightforward approach for constructing Lyapunov functions for first-order methods applied to quadratic functions. Our approach involves bringing the iteration matrix to an upper triangular form using Schur decomposition, then examining the value of the last coordinate of the state vector. This value is multiplied by a magnitude smaller than one at each iteration. Consequently, this value should decrease at each iteration, provided that the method converges. We rigorously prove the suitability of this Lyapunov function for all first-order methods and derive the necessary conditions for the proposed function to decrease monotonically. Experiments conducted with general convex functions are also presented, alongside a study on the limitations of the proposed approach.

    Remarkably, the newly discovered Lyapunov function is straightforward and does not explicitly depend on the exact method formulation or function characteristics like strong convexity or smoothness constants. In essence, a single expression serves as a Lyapunov function for several methods, including Heavy Ball, Nesterov Accelerated Gradient, and Triple Momentum, among others. To the best of our knowledge, this approach has not been previously reported in the literature.
\end{abstract}

\section{Introduction}
Consider the following quadratic optimization problem:
\begin{equation}
    \label{problem}
    \min\limits_{x \in \mathbb{R}^d} f(x) =  \min\limits_{x \in \mathbb{R}^d} \dfrac{1}{2} x^T W x - b^T x + c, \text{ where }W \in \mathbb{S}^d_{++}
\end{equation}
The problem itself is equivalent to solving a linear system with a positive definite matrix. However, this problem is not the main object of study. In the example of solving this problem, we want to study a specific way of constructing the \textit{Lyapunov function} for the optimization methods. It is quite common for the convergence of accelerated first-order methods to be non-monotonic, even for quadratic objectives \cite{goh2017why, giselsson2014monotonicity} (see the numerical example in Figure~\ref{oplyap:fig:hb_non_monotonic}). These methods include the Polyak Heavy Ball method \cite{polyak1964some}, Nesterov's accelerated gradient method \cite{nesterov2003introductory}, the triple momentum method \cite{van2017fastest}, and others. It is important to note that these methods are widely used in modern neural network training, including in the currently popular large language models. The quadratic problem appears not only in usual linear least squares problems but also in consensus search problems in decentralized systems \cite{gorbunov2022recent}. Therefore, studying their convergence is of great significance.

Lyapunov functions have proven to be invaluable tools for understanding the stability properties of optimization algorithms. They offer a systematic approach for assessing convergence, boundedness, and even optimality in the iterative optimization process. These functions capture the behavior of an algorithm by assigning a scalar quantity to each iteration, enabling a broader understanding of the underlying dynamics and convergence behavior. One may think of them as energy functions that should not increase during the optimization trajectory of a given method. The main objective of this paper is to present a simple approach for constructing Lyapunov functions for such methods. For example, the following function will monotonically decrease for the Heavy Ball and Nesterov Accelerated Gradient methods applied to \eqref{problem} (assuming that $x^*$ is the solution and optimal hyperparameters were chosen):

\begin{equation}
    \label{oplyap:eq:base}
    V(x_k, x_{k-1}, x_{k-2}) = \Vert x_{k-1} - x^*\Vert ^2 - \langle x_k - x^*, x_{k-2} - x^* \rangle
\end{equation}

\begin{figure}[h!]
    \centering
    \includegraphics[width=\linewidth]{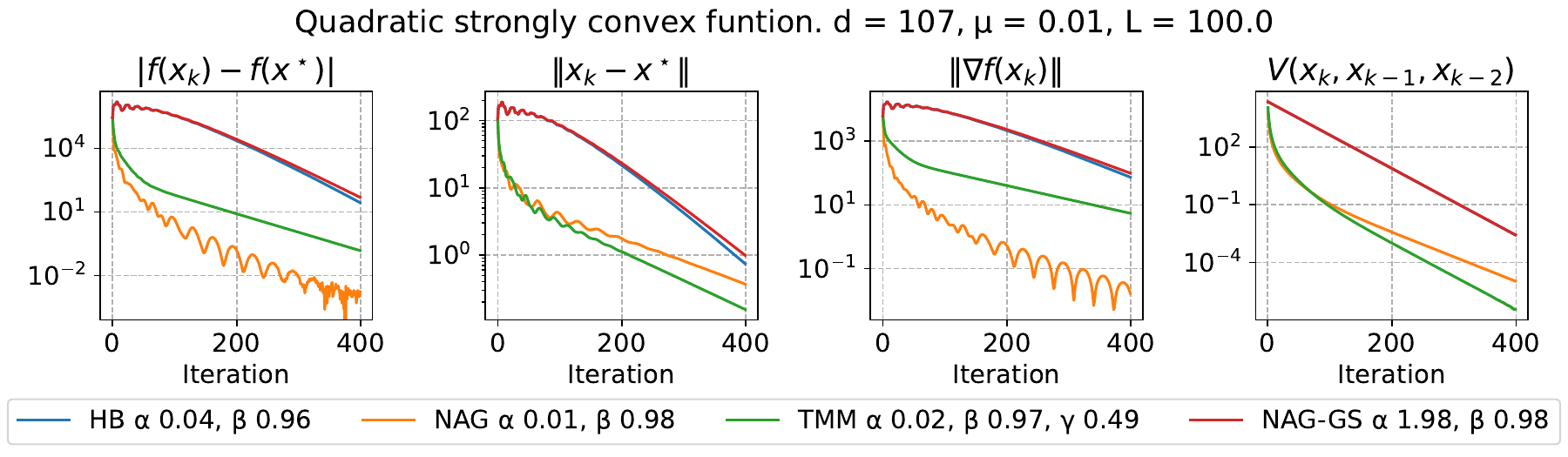}
    \caption{Dynamics of several accelerated methods with optimal hyperparameters $\alpha, \beta, \gamma$ applied to the general strongly convex quadratic problem \eqref{problem} with dimension $d = 107$ are presented. Here and later, we use the following notation: \textbf{HB} - Heavy Ball \cite{polyak1964some}, \textbf{NAG} - Nesterov Accelerated Gradient \cite{nesterov1983method}, \textbf{TMM} - Triple Momentum Method \cite{van2017fastest}, \textbf{NAG-GS} - Nesterov Accelerated Gradient with Gauss-Seidel Splitting \cite{leplat2022nag}. It is easy to see that the usual metrics on the first three subfigures on the left do not demonstrate a monotonic decrease, while the proposed Lyapunov function $V(x_k, x_{k-1}, x_{k-2})$ from \eqref{oplyap:eq:base} works for all these methods.}
    \label{oplyap:fig:hb_non_monotonic}
\end{figure}

Recent advances in discovering Lyapunov functions can be found in the book \cite{vorontsova2021convex}, which covers a much more general setup of minimizing general (possibly non-quadratic) $\mu$-strongly convex functions. Sometimes, in the context of accelerated methods, the Lyapunov function is also referred to as the \textit{potential} function \cite{d2021acceleration}, especially when the problem is convex but not strongly convex. Important results have been obtained regarding the provably fastest Lyapunov functions \cite{taylor2018lyapunov} over a parameterized family of functions (called Lyapunov function candidates) for general strongly convex smooth functions. The parametric class of Lyapunov functions considered were quadratic functions. The class of minimized functions was $\mu$-strongly convex with $L$-Lipschitz gradient. However, constructing such functions involves solving a small but rather complex SDP problem. A fair comparison of the functions from \cite{taylor2018lyapunov} and a function presented in the current work is a topic for further research.

\section{Illustrative example: Heavy Ball method}

If we have a first-order method applied to a quadratic function, which operates on the current point $ x_k $ (note that the gradient at the current point is also expressed in this term $ \nabla f(x_k) = W x_k - b $) and the previous point $ x_{k-1} $, it can be formulated as follows:
\begin{equation}
    \label{oplyap:dynsys}
    z_{k+1} = Mz_k,
\end{equation}
where $ z_k $ is a state vector (for example, $ z_{k} = (x_k, x_{k-1}) $). The system of equation \eqref{oplyap:dynsys} represents a linear dynamical system. The idea behind constructing a Lyapunov function involves creating a positive quantity that decreases along the trajectories of the dynamical system.

The presence of such a function ensures the convergence of the dynamical system. It should be noted that for the optimization problem of a function of general form (non-quadratic), it is not possible to write down the iteration of the method with the help of a dynamic linear system.

Consider the Heavy Ball method \cite{polyak1964some}.
\begin{equation}\label{oplyap:polyak}
    x_{k+1} = x_k - \alpha \nabla f(x_k) + \beta (x_k - x_{k-1}).
\end{equation}
It is well-known \cite{polyak1964some, ghadimi2015global} that $ \alpha $ and $ \beta $ can be selected in such a way that the Heavy Ball method achieves global convergence. This can be demonstrated by finding a suitable \emph{Lyapunov function} $ V(x_k, x_{k-1}) $ that decays along the iterations.

Our goal is to develop a systematic way to propose candidates for Lyapunov functions for the method \eqref{oplyap:dynsys}. Table~\ref{oplyap:tab:two_step_methods} illustrates how different methods can be expressed in this way and how the iteration matrix $ M $ appears in each particular case. The idea can be applied to any standard optimization method, but we will illustrate how it works for the Polyak method and later generalize the idea to any method presented in the form of \eqref{oplyap:dynsys} under some conditions on the iteration matrix $ M $.

\subsection{Reduction to a scalar case}
\label{oplyap:sec:reduction_to_a_scalar}
For the sake of clarity, we will start with the case when $f(x)$ is a strongly convex quadratic function with the minimum at $x^\star = 0$, thus it could be written as 
\begin{equation}
 \label{oplyap:eq:original_problem}
 \begin{split}
 f(x) = & \frac12 \langle Wx, x\rangle \to \min_{x \in \mathbb{R}^d},\\ 
 W &\in \mathbb{S}^d_{++}.
 \end{split}
\end{equation}

Where $W \in \mathbb{S}^d_{++}$ means, that matrix $W$ is symmetric positive definite (SPD). Moreover, we will assume the following characteristics about it: $0 < \mu = \lambda_{min}(W)$; $\lambda_{max}(W) = L$. Note, that $\mu$ is a strong convexity constant, while $L$ - is the Lipschitz constant for the gradient of the function.

Since $W$ is an SPD matrix, it has the eigendecomposition $W = Q \Lambda Q^*$, where $Q$ is orthogonal and $\Lambda$ is diagonal. If we assign new variables 
\begin{equation}
 \label{change_var}
 \hat{x}_{k} = Q^* x_k.
\end{equation}

The function will look like:

\begin{equation}
f(\hat{x}) = \frac12 \langle W x, x \rangle = \frac12 \langle Q \Lambda Q^* x, x \rangle = \frac12 \langle \Lambda Q^* x , Q^* x \rangle = \frac12 \langle \Lambda\hat{x}, \hat{x}\rangle
\end{equation}

Taking into account, that the gradient is $\nabla f(\hat{x}) = \Lambda \hat{x}$, the method \eqref{oplyap:polyak} will be written as follows
\begin{equation}\label{oplyap:lin}
 \hat{x}_{k+1} = \hat{x}_k - \alpha \Lambda \hat{x}_k + \beta (\hat{x}_k - \hat{x}_{k-1})\qquad
 \hat{x}_{k+1} = (I - \alpha \Lambda + \beta I) \hat{x}_k - \beta \hat{x}_{k-1}.
\end{equation}
We can now use the common reformulation:
\begin{equation}
 \begin{cases}
 \hat{x}_{k+1} &= (I - \alpha \Lambda + \beta I) \hat{x}_k - \beta \hat{x}_{k-1} \\
 \hat{x}_{k} &= \hat{x}_k,
 \end{cases}
\end{equation}

Let’s use the following notation $ \hat{z}_k = \begin{bmatrix} 
 \hat{x}_{k+1} \\
 \hat{x}_{k}
 \end{bmatrix}$. Therefore $\hat{z}_{k+1} = M \hat{z}_k$, where the iteration matrix $M$ is:

\begin{equation}
 M = \begin{bmatrix} 
 I - \alpha \Lambda + \beta I & - \beta I \\
 I & 0_{d}
 \end{bmatrix}
\end{equation}
Note, that $M$ is $2d \times 2d$ matrix with 4 block-diagonal matrices of size $d \times d$ inside. It means, that we can rearrange the order of coordinates to make $M$ block-diagonal in the following form (see Figure~\ref{oplyap:fig:rearrangement}). Note that in the equation below, the matrix $M$ denotes the same as in the notation above, except for the described permutation of rows and columns. We use this slight abuse of notation for the sake of clarity. 
\begin{figure}[h!]
 \centering
 \begin{minipage}{0.35\linewidth}
 \centering
 \includegraphics[width=\textwidth]{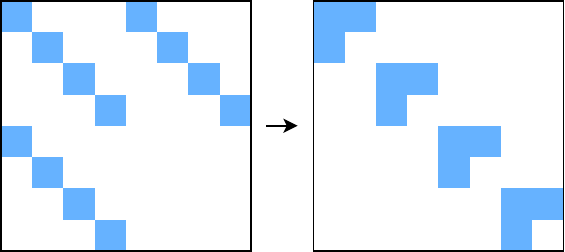}
 \caption{Illustration of matrix $ M $ rearrangement}
 \label{oplyap:fig:rearrangement}
 \end{minipage}
 \hfill
 \begin{minipage}{0.6\linewidth}
 \centering
 \begin{equation}
 \begin{aligned}
 \begin{bmatrix} 
 \hat{x}_{k}^{(1)} \\
 \vdots \\
 \hat{x}_{k}^{(d)} \\
 \addlinespace 
 \hat{x}_{k-1}^{(1)} \\
 \vdots \\
 \hat{x}_{k-1}^{(d)}
 \end{bmatrix} \to 
 \begin{bmatrix} 
 \hat{x}_{k}^{(1)} \\
 \addlinespace 
 \hat{x}_{k-1}^{(1)} \\
 \vdots \\
 \hat{x}_{k}^{(d)} \\
 \addlinespace 
 \hat{x}_{k-1}^{(d)}
 \end{bmatrix} \quad M = \begin{bmatrix}
 M_1\\
 &M_2\\
 &&\ldots\\
 &&&M_d
 \end{bmatrix}
 \end{aligned}
 \label{oplyap:eq:rearrangement}
 \end{equation}
 \end{minipage}
\end{figure}
where $\hat{x}_{k}^{(i)}$ is $i$-th coordinate of vector $\hat{x}_{k} \in \mathbb{R}^d$ and $M_i$ stands for $2 \times 2$ matrix. This rearrangement allows us to study the dynamics of the \eqref{oplyap:polyak} independently for each dimension. One may observe, that the asymptotic convergence rate of the $2d$-dimensional vector sequence of $\hat{z}_k$ is defined by the worst convergence rate among its block of coordinates. Thus, it is enough to study the optimization in a one-dimensional case. For $i$-th coordinate with $\lambda_i$ as an $i$-th eigenvalue of matrix $W$ we have: 
\begin{equation}
 M_i = \begin{bmatrix} 
 1 - \alpha \lambda_i + \beta & -\beta \\
 1 & 0
 \end{bmatrix}
 \label{oplyap:eq:m_i}
\end{equation}
\subsection{Idea}
From this moment we will consider the scalar case. Here we have the problem
\begin{equation}
 f(x) = \frac\lambda2 x^2 \to \min_{x \in \mathbb{R}^1}, \quad \lambda >0
\end{equation}
It is worth mentioning, that we still keep in mind the original problem \eqref{oplyap:eq:original_problem}, when $\lambda$ is some eigenvalue of initial matrix $W$, which means, that $0 < \mu \leq \lambda \leq L$. The iteration takes the form
\begin{equation}
 x_{k+1} = x_k - \alpha \lambda x_k + \beta(x_k - x_{k-1})
 \qquad
 \begin{cases}
 x_{k+1} &= (1 - \alpha \lambda) x_k - \beta x_{k-1} \\
 x_{k} &=x_k ,
 \end{cases}
 \end{equation}
which can be rewritten in the matrix form as
\begin{equation}
 \label{oplyap:scalar_problem}
 z_{k+1} = M z_k, \quad M = \begin{bmatrix} 
 1 - \alpha \lambda + \beta & -\beta \\
 1 & 0
 \end{bmatrix},\quad
 z_k = \begin{bmatrix} 
 x_{k} \\
 x_{k-1}
 \end{bmatrix}
\end{equation}
The method will be convergent if $\rho(M) < 1$, and the optimal parameters can be computed by optimizing the spectral radius \cite{polyak1964some}
\begin{equation}
 \label{oplyap:eq:optimal_hyperpars}
 \alpha^*, \beta^* = \arg \min_{\alpha, \beta} \max_{\lambda \in [\mu, L]} \rho(M) \quad \alpha^* = \dfrac{4}{(\sqrt{L} + \sqrt{\mu})^2}; \quad \beta^* = \left(\dfrac{\sqrt{L} - \sqrt{\mu}}{\sqrt{L} + \sqrt{\mu}}\right)^2
\end{equation}
It can be shown \eqref{oplyap:eq:optimal_eigs} that for such parameters the matrix $M$ has complex eigenvalues, which forms a conjugate pair, so the distance to the optimum (in this case, $\Vert z_k \Vert$), generally, will not go to zero monotonically. 
One can use the machinery of the matrix Lyapunov equation \cite{hammarling1982numerical} to derive the candidates for the Lyapunov function, which will be quadratic. However, there is a simpler way. 
\subsubsection{Lyapunov function formulation}
Consider Schur decomposition of the matrix $M:$
\begin{equation}
 \label{oplyap:eq:schur}
 M = U T U^*; \quad U^*U = I; \quad T = \begin{bmatrix} 
 t_{11} & t_{12} \\
 0 & t_{22}
 \end{bmatrix}
\end{equation}
where $U$ is a unitary matrix, and $T$ is a complex upper triangular matrix. If we manage to find $T$, we will have an easy iteration analysis. From the \eqref{oplyap:scalar_problem}
\begin{equation}
 \label{oplyap:eq:iteration_analysis}
 \begin{split}
 z_{k+1} &= M z_k = U T U^* z_k \\
 U^* z_{k+1}&= T U^* z_k \\
 w_{k+1} &= T w_k,
 \end{split}
\end{equation}
where the substitution $w_k = U^* z_k$ was introduced.
Since $T$ is upper triangular, the last element of the vector $w_k$ is just multiplied by the same number $t_{22}$ at each iteration:
\begin{equation}
 \label{oplyap:eq:upper_triangle_evolution}
 \begin{bmatrix} 
 (w_{k+1})_1 \\
 (w_{k+1})_2
 \end{bmatrix} = \begin{bmatrix} 
 t_{11} & t_{12} \\
 0 & t_{22}
 \end{bmatrix} \begin{bmatrix} 
 (w_{k})_1 \\
 (w_{k})_2
 \end{bmatrix} = \begin{bmatrix} 
 t_{11} & t_{12} \\
 0 & t_{22}
 \end{bmatrix}^k \begin{bmatrix} 
 (w_{0})_1 \\
 (w_{0})_2
 \end{bmatrix} \; \rightarrow \; (w_{k+1})_2 = (t_{22})^k ( w_{0})_2
\end{equation}

In equations \eqref{oplyap:eq:upper_triangle_evolution}, \eqref{oplyap:eq:lyapunov} the term $(m)_i$ denotes the $i$-th coordinate of the vector $m$. In this form, it’s obvious, that if the method converges, the absolute value of $t_{22}$ is lower than $1$. Therefore, we can pick the absolute value of the $w_{k}$ as the Lyapunov function, because it is the converging geometric progression with the convergence rate $\vert t_{22}\vert$. Since the diagonal elements $t_{11}, t_{22}$ are the eigenvalues of the matrix $M$, the rate of convergence is the spectral radius of the iteration matrix, which makes the proposed Lyapunov function asymptotically optimal for the particular method. 
\begin{equation}
 \label{oplyap:eq:lyapunov}
 V(x_{k}, x_{k-1}) = | (w_{k})_2|^2 = | (U^* z_{k})_2|^2 = \left\vert \left(U^*\begin{bmatrix} x_k \\ x_{k-1} \end{bmatrix}\right)_2 \right\vert^2
\end{equation}

\subsubsection{Explicit Schur decomposition of the iteration matrix $M$}

The diagonal elements of the matrix $T$ are the eigenvalues of $M$. At the same time, the first column of the matrix $U$ is the eigenvector of the matrix $M$. Let us obtain the expression for it. It is easy to verify, that the eigenvalues of the iteration matrix are

\begin{equation}
 t_{11}, t_{22} = \lambda^M_1, \lambda^M_2 = \lambda \left( \begin{bmatrix} 
 1 - \alpha \lambda + \beta & -\beta \\
 1 & 0
 \end{bmatrix}\right) = \dfrac{1+\beta - \alpha \lambda \pm \sqrt{(1+\beta - \alpha\lambda)^2 - 4\beta}}{2}
\end{equation}

If we'll use the optimal values $\alpha^*, \beta^*$ from \eqref{oplyap:eq:optimal_hyperpars}, assuming $\mu > 0$ and $L>0$, we have:

\begin{equation}
 \label{oplyap:eq:optimal_eigs}
 t_{11}, t_{22} = \dfrac{\mu + L - 2\lambda \pm 2\sqrt{(L - \lambda)(\mu - \lambda)}}{(\sqrt{L} + \sqrt{\mu})^2}
\end{equation}

Let us also verify, that the vector $ (\lambda^M \; 1)^T$ will be the (unnormalized) eigenvector for iteration matrix $M$. Here we denote $\lambda^M$ as any eigenvalue of matrix $M$ (either $\lambda^M_1$ or $\lambda^M_2$). We have
\begin{equation*}
 \begin{bmatrix}
 1 - \alpha \lambda + \beta & -\beta \\
 1 & 0
 \end{bmatrix} 
\begin{bmatrix}
 \lambda^M \\ 
 1
 \end{bmatrix} = \lambda^M\begin{bmatrix}
 \lambda^M \\ 
 1
 \end{bmatrix} \to \lambda^M = \dfrac{1+\beta - \alpha \lambda \pm \sqrt{(1+\beta - \alpha\lambda)^2 - 4\beta}}{2}
\end{equation*}

To build matrix $U$ from \eqref{oplyap:eq:schur} we can take vector $u_1$ as an eigenvector of $M$:

\begin{equation}
 U = \begin{bmatrix}
 u_1 \; u_2
 \end{bmatrix}
 \quad
 u_1~=~\dfrac{1}{\sqrt{1 + (\lambda^M)^*\lambda^M}}\begin{bmatrix}
 \lambda^M \\ 1
 \end{bmatrix},
\end{equation}
while the second vector $u_2$ can be taken as an orthogonal vector as $ u_2 \sim \begin{bmatrix}
 1 \\ -(\lambda^M)^*
\end{bmatrix}$. Thus, we can write down the matrix $U$:

\begin{equation}
 U = \dfrac{1}{\sqrt{1 + (\lambda^M)^*\lambda^M}}\begin{bmatrix} \lambda^M & 1 \\ 1 & -(\lambda^M)^* \end{bmatrix}
\end{equation}

Note, that if the eigenvalues $\lambda_1^M, \lambda_2^M$ is not a conjugate pair, we couldn't write down the expression \eqref{oplyap:eq:schur_decomposition_ex}. The term \textit{conjugate pair} refers to either complex eigenvalues, which satisfy $(\lambda_1^M)^* = \lambda_2^M$ or real equal eigenvalues $\lambda_1^M = \lambda_2^M$. Taking into account, that eigenvalues of $M$ is a conjugate pair, the iteration matrix will take the form

\begin{equation}
 \label{oplyap:eq:schur_decomposition_ex}
 M = 
\underbrace{\dfrac{1}{\sqrt{1 + \lambda^M_1\lambda^M_2}}\begin{bmatrix} \lambda^M_1 & 1 \\ 1 & -\lambda^M_2 \end{bmatrix} }_{U}
\underbrace{\begin{bmatrix} \lambda^M_1 & * \\ 0 & \lambda^M_2 \end{bmatrix}\vphantom{\dfrac{1}{\sqrt{1 + \lambda^M_1\lambda^M_2}}}}_T
\underbrace{\dfrac{1}{\sqrt{1 + \lambda^M_1\lambda^M_2}}\begin{bmatrix} \lambda^M_2 & 1 \\ 1 & -\lambda^M_1 \end{bmatrix}}_{U^*}
\end{equation}

The diagonal elements of the matrix $T$ are the eigenvalues of $M$, so $t_{22} = \lambda^M$. Returning to \eqref{oplyap:eq:lyapunov}, we have:
\begin{equation}
 \begin{split}
 V(x_{k}, x_{k-1}) &= \left\vert \dfrac{1}{\sqrt{1 + \lambda^M_1\lambda^M_2}}\begin{bmatrix} 1 & -\lambda^M_1 \end{bmatrix} \begin{bmatrix} x_k \\ x_{k-1} \end{bmatrix} \right\vert^2 \\
 &= \dfrac{1}{1 + \lambda^M_1\lambda^M_2} \left\vert x_k - \lambda^M_1 x_{k-1}\right\vert^2 \\
 &= \dfrac{1}{1 + \lambda^M_1\lambda^M_2} \left(\text{Re}^2(x_k - \lambda^M_1 x_{k-1}) + \text{Im}^2(x_k - \lambda^M_1x_{k-1}) \right) \\
 &= \dfrac{1}{1 + \lambda^M_1\lambda^M_2} \left(\left(x_k - \text{Re}(\lambda^M_1) x_{k-1}\right)^2 + \left(\text{Im}( \lambda^M_1) x_{k-1}\right) ^2\right) \\
 &= \dfrac{1}{1 + \lambda^M_1\lambda^M_2} \left(x_k^2 - 2 \text{Re}(\lambda^M_1) x_k x_{k-1} + \text{Re}^2(\lambda^M_1)x_{k-1}^2 + \text{Im}^2( \lambda^M_1) x_{k-1}^2\right) \\
 &= \dfrac{1}{1 + \lambda^M_1\lambda^M_2} \left(x_k^2 - 2 \text{Re}(\lambda^M_1) x_k x_{k-1} + \vert \lambda^M_1\vert^2 x_{k-1}^2\right) \\
 \end{split}
\end{equation}

\subsubsection{Optimal hyperparameters for the method and the spectrum of the iteration matrix}

Now we’ll consider the eigenvalues of the iteration matrix $M$. We’ll start with the optimal hyperparameters $\alpha^*, \beta^*$ from \eqref{oplyap:eq:optimal_hyperpars}:

\begin{equation}
 \text{Re}(\lambda^M_1) = \dfrac{L + \mu - 2\lambda}{(\sqrt{L} + \sqrt{\mu})^2}; \quad \text{Im}(\lambda^M_1) = \dfrac{\pm 2\sqrt{(L - \lambda)(\lambda - \mu)}}{(\sqrt{L} + \sqrt{\mu})^2}; 
\end{equation}

\begin{equation}
 \vert \lambda \vert = \dfrac{L - \mu}{(\sqrt{L} + \sqrt{\mu})^2} \quad \vert \lambda \vert^2 = \dfrac{(L - \mu)^2}{(\sqrt{L} + \sqrt{\mu})^4}
\end{equation}

\begin{equation}
 \begin{split}
 V(x_k, x_{k-1})&= \dfrac{1}{1 + \lambda^M_1\lambda^M_2} \left(x_k^2 - 2 \dfrac{L + \mu - 2\lambda}{(\sqrt{L} + \sqrt{\mu})^2} x_k x_{k-1} + \dfrac{(L - \mu)^2}{(\sqrt{L} + \sqrt{\mu})^4}x_{k-1}^2\right) \\
 &= \dfrac{1}{1 + \lambda^M_1\lambda^M_2} \left( \left( x_k - 2 \dfrac{L + \mu - 2\lambda}{(\sqrt{L} + \sqrt{\mu})^2} x_{k-1}\right)x_k + \dfrac{(L - \mu)^2}{(\sqrt{L} + \sqrt{\mu})^4}x_{k-1}^2\right) \\
 &= \dfrac{1}{1 + \lambda^M_1\lambda^M_2} \left( -\dfrac{(\sqrt{L} - \sqrt{\mu})^2}{(\sqrt{L} + \sqrt{\mu})^2}x_{k-2}x_k + \dfrac{(\sqrt{L} + \sqrt{\mu})^2(\sqrt{L} - \sqrt{\mu})^2}{(\sqrt{L} + \sqrt{\mu})^4}x_{k-1}^2\right) \\
 &= \dfrac{1}{1 + \lambda^M_1\lambda^M_2} \dfrac{(\sqrt{L} - \sqrt{\mu})^2}{(\sqrt{L} + \sqrt{\mu})^2} \left( x_{k-1}^2 - x_k x_{k-2}\right) \\
 \end{split}
\end{equation}
Overall, we have the simpler formula for the Lyapunov function in this setting:

\begin{equation}
 \label{oplyap:eq:basest}
 V(x_k, x_{k-1}) = x_{k-1}^2 - x_k x_{k-2}
\end{equation}

Let us highlight what was shown at the moment. We considered the problem \eqref{oplyap:eq:original_problem} and the heavy ball method \eqref{oplyap:polyak} applied to it. It was shown, that the dynamics happen independently on each coordinate with its iteration matrix $M_i$ (which was denoted in this section as $M$ for simplicity) \eqref{oplyap:eq:m_i} for each dimension. Then, we brought the iteration matrix $M_i$ to the upper triangular form using Schur decomposition. We demonstrated, that for each coordinate the proposed expression \eqref{oplyap:eq:basest} is monotonically decreasing during the iteration procedure under the condition, that $\lambda_1^M$ and $\lambda_2^M$ are conjugate pair. 

It can be shown (see Figure~\ref{oplyap:fig:non_optimal}), that for \textbf{HB}, \textbf{NAG} and \textbf{NAG-GS} with optimal parameters we have a spectrum, where for each dimension eigenvalues are conjugate pairs, while for the \textbf{TMM} this is not true, which means, that generally there is no guarantee, that proposed function \eqref{oplyap:eq:basest} will serve as a Lyapunov function for this method. However, occasionally, sometimes it works even in this case, but we will show experiments, where $V(x_k, x_{k-1}, x_{k-2})$ will not monotonically decrease for \textbf{TMM} even in quadratic case.

However, optimal hyperparameter setting requires the knowledge of $\mu$ and $L$, which is not always possible in practice. Therefore, it is even more interesting, that we can derive convergence properties straightforward from the spectrum of the iteration matrix and verify, that some set of hyperparameters ensures proposed $V(x_k, x_{k-1}, x_{k-2})$ to be the Lyapunov function (see Figure~\ref{oplyap:fig:non_optimal}).

\begin{table}[h!]
 \centering
 \resizebox{\textwidth}{!}{\begin{tabular}{@{}ccccc@{}}
 \toprule
 Method & $a$ & $b$ & $\rho(M)$ & $\lambda_1^M$ and $\lambda_2^M$ are the conjugate pairs if \\ 
 \midrule
 \begin{tabular}[c]{@{}c@{}}\textbf{HB} \cite{polyak1964some}\\ with $\alpha^\star, \beta^\star$\end{tabular} & $\dfrac{2(L + \mu - 2\lambda)}{(\sqrt{L} + \sqrt{\mu})^2}$ & $-\dfrac{(\sqrt{L} - \sqrt{\mu})^2}{(\sqrt{L} + \sqrt{\mu})^2}$ & $\dfrac{\sqrt{L} - \sqrt{\mu}}{\sqrt{L} + \sqrt{\mu}}$ & always \\
 \begin{tabular}[c]{@{}c@{}}\textbf{NAG}\cite{nesterov2018lectures}\\ with $\alpha^\star, \beta^\star$\end{tabular} & $\dfrac{2\sqrt{L}(L-\lambda)}{(\sqrt{L} + \sqrt{\mu})L}$ & $-\dfrac{(\sqrt{L} - \sqrt{\mu})(L-\lambda)}{(\sqrt{L} + \sqrt{\mu})L}$ & $1 - \sqrt{\dfrac{\mu}{L}}$ & always \\
 \begin{tabular}[c]{@{}c@{}}\textbf{TMM}\cite{van2017fastest}\\ with $\alpha^\star, \beta^\star,\gamma^\star$\end{tabular} & $\frac{2L + \mu - 3 \lambda - (L-\lambda) \sqrt{\frac{\mu}{L}}}{L\left(1 + \sqrt{\frac{\mu}{L}}\right)}$ & $-\dfrac{(\sqrt{L} - \sqrt{\mu})^2(L - \lambda)}{(\sqrt{L} + \sqrt{\mu})L}$ & $\left(1 - \sqrt{\dfrac{\mu}{L}}\right)^{\frac32}$ & not $\forall \lambda$ \\
 \begin{tabular}[c]{@{}c@{}}\textbf{NAG-GS}\cite{leplat2022nag}\\ with $\alpha^\star, \beta^\star$\end{tabular} & $\frac{4 \left( \mu + \sqrt{\mu L} \right) \left( -\lambda \left( \sqrt{\frac{L}{\mu}} + 1 \right) + \sqrt{\mu L} + L \right)}{\left( \mu + 2 \sqrt{\mu L} + L \right)^2}$ & $-\frac{(L - \mu)^2}{(L + \mu + 2 \sqrt{\mu L})^2}$ & $\frac{L - \mu}{L + \mu + 2 \sqrt{\mu L}}$ & always \\
 \bottomrule
 \end{tabular}}
 \caption{Reformulation of first order methods with optimal parameters in the format given in theorem \ref{oplyap:th:scalar_lyapunov}}
\end{table}

\section{Lyapunov function for first-order methods for quadratic function}

\subsection{Scalar case}
As soon as we formulate the result for the Heavy Ball method with optimal hyperparameters, we can generalize the idea to an arbitrary two-step method, which is convergent and has a conjugate pair of eigenvalues of the iteration matrix. This result is formulated in the following theorem.

\begin{theorem}
 \label{oplyap:th:scalar_lyapunov}
 For the quadratic optimization problem in the form of
 \[
 f(x) = \frac\lambda2 x^2 \to \min_{x \in \mathbb{R}^1}, \quad \lambda >0
 \]
 Given any convergent optimization method, which could be written in the following form
 \begin{align}
 \label{oplyap:eq:two_step_scalar}
 x_{k+1} = a x_k + b x_{k-1}
 \end{align}
 , where $a^2 + 4b \leq 0$ it has the following Lyapunov function:
 \[
 V(x_k, x_{k-1}, x_{k-2}) = x_{k-1}^2 - x_k x_{k-2} 
 \]
\end{theorem}

\begin{table}[t!]
 \centering
 \begin{tabular}{@{}cc@{}}
 \toprule
 Method & Iteration \\ \midrule
 \textbf{HB} \cite{polyak1964some} & \begin{tabular}{@{}c@{}}
 $x_{k+1} = x_k - \alpha\nabla f(x_k) + \beta(x_k - x_{k-1})$ \\
 $\mathbf{x_{k+1}} = (1+\beta - \alpha\lambda)\mathbf{x_k} - \beta \mathbf{x_{k-1}}$
 \end{tabular} \\ \midrule
 \textbf{NAG}\cite{nesterov1983method} & \begin{tabular}{@{}c@{}}
 $\begin{cases}y_{k+1} = x_k + \beta (x_k - x_{k-1}) \\ x_{k+1} = y_{k+1} - \alpha \nabla f(y_{k+1}) \end{cases}$ \\
 $\mathbf{x_{k+1}} = (1+\beta)\left(1 - \alpha \lambda\right)\mathbf{x_k} - \beta\left(1 - \alpha \lambda\right)\mathbf{x_{k-1}}$
 \end{tabular} \\ \midrule
 \textbf{TMM}\cite{van2017fastest} & \begin{tabular}{@{}c@{}}
 $x_{k+1} = (1+\beta)x_k - \beta x_{k-1} -\alpha \nabla f\left( (1 + \gamma)x_k - \gamma x_{k-1}\right)$ \\
 $\mathbf{x_{k+1}} = (1+ \beta - \alpha (1 + \gamma) \lambda)\mathbf{x_k} + \left(\alpha \gamma \lambda - \beta \right)\mathbf{x_{k-1}} $
 \end{tabular} \\ \midrule
 \textbf{NAG-GS}\cite{leplat2022nag} & \begin{tabular}{@{}c@{}}
 $\begin{cases}y_{k} = \beta y_{k-1} + (1-\beta)x_k - \alpha \nabla f(x_k)\\ x_{k+1} = \beta x_{k} + (1-\beta)y_k \end{cases}$ \\
 $\mathbf{x_{k+1}} = \left((2 \beta + (1-\beta)^2) - \alpha (1-\beta) \lambda\right) \mathbf{x_k} - \beta^2 \mathbf{x_{k-1}}$
 \end{tabular} \\ \bottomrule
 \end{tabular}
 \caption{Correspondence between accelerated first-order methods and two-step notation given in \eqref{oplyap:eq:two_step_scalar}. Notation $\mathbf{x_k}$ made only for the sake of clarity.}
 \label{oplyap:tab:two_step_methods}
\end{table}

\begin{proof}
\begin{enumerate}
 \item Clearly, the method could be written in the form:
 $$
 \begin{bmatrix} 
 x_{k+1} \\
 x_k 
 \end{bmatrix} = M \begin{bmatrix} 
 x_{k} \\
 x_{k-1}
 \end{bmatrix}, \qquad
 M = \begin{bmatrix} 
 a & b \\
 1 & 0
 \end{bmatrix}
 $$
 While the eigenpairs are $\lambda^M = \lambda^M_1, \lambda^M_2$ or $\lambda^M = \dfrac{a \pm \sqrt{a^2 + 4b}}{2}$ and $v^M = \begin{bmatrix} 
 \lambda^M \\
 1
 \end{bmatrix}$
 \item Thus, we can explicitly construct Schur decomposition:
 \begin{equation}
 M = 
 \underbrace{\dfrac{1}{\sqrt{1 + \left(\lambda^M_1\right)^*\lambda^M_1}}\begin{bmatrix} \lambda^M_1 & 1 \\ 1 & -\left(\lambda^M_1\right)^* \end{bmatrix} }_{U}
 \underbrace{\begin{bmatrix} \lambda^M_1 & * \\ 0 & \lambda^M_2 \end{bmatrix}\vphantom{\dfrac{1}{\sqrt{1 + \lambda^M_1\lambda^M_2}}}}_T
 \underbrace{\dfrac{1}{\sqrt{1 + \left(\lambda^M_1\right)^*\lambda^M_1}}\begin{bmatrix} \left(\lambda^M_1\right)^* & 1 \\ 1 & -\lambda^M_1 \end{bmatrix}}_{U^*}
 \end{equation}
    
 Taking into account convergence condition: $\rho(M) < 1$, i.e. $\max\left(|\lambda^M|\right)~<~1$ and the imaginary spectrum of the iteration matrix $a^2 + 4b \leq 0$, we will write down explicitly:
 \begin{itemize}
 \item $\left(\lambda^M_1\right)^* = \lambda^M_2$
 \item $|\lambda^M| = \sqrt{-b}$
 \item $|\lambda^M|^2 = -b$
 \item $\rho(M) = \sqrt{-b}$
 \item $\text{Re}(\lambda^M) = \frac{a}{2}$
 \item $\text{Im}(\lambda^M) = \frac{\sqrt{-a^2 - 4b}}{2}$
 \item Similarly to \eqref{oplyap:eq:upper_triangle_evolution} we can say, that in the new variables ($w_k$), the absolute value of the last coordinate will monotonically decrease. Using the reformulation $z_{k+1} = M z_k \; w_k = U^* z_k; \; w_{k+1} = T w_k$ and taking into account, that $\left(\lambda_1^M\right)^* = \lambda_2^M$ the Lyapunov function will take the following form:
 \begin{equation}
 \begin{split}
 V(\cdot) &= | (w_{k})_2|^2 = | (U^* z_{k})_2|^2 = \left\vert \left(U^*\begin{bmatrix} x_k \\ x_{k-1} \end{bmatrix}\right)_2 \right\vert^2 \\
 &=\dfrac{1}{1 + \lambda^M_1\lambda^M_2} \left\vert x_k - \lambda^M_1 x_{k-1}\right\vert^2 \\
 &\simeq \text{Re}^2(x_k - \lambda^M_1 x_{k-1}) + \text{Im}^2(x_k - \lambda^M_1x_{k-1}) \\
 &= \left(x_k - \text{Re}(\lambda^M_1) x_{k-1}\right)^2 + \left(\text{Im}( \lambda^M_1) x_{k-1}\right) ^2 \\
 &= x_k^2 - 2 \text{Re}(\lambda^M_1) x_k x_{k-1} + \text{Re}^2(\lambda^M_1)x_{k-1}^2 + \text{Im}^2( \lambda^M_1) x_{k-1}^2 \\
 &= x_k^2 - 2 \text{Re}(\lambda^M_1) x_k x_{k-1} + \vert \lambda^M_1\vert^2 x_{k-1}^2
 \end{split}
 \end{equation}
 \item As soon as $x_k = a x_{k-1} + b x_{k-2}$, we can write the Lyapunov function \begin{equation}
 \begin{split}
 V(\cdot) &= x_k^2 - 2 \text{Re}(\lambda^M_1) x_k x_{k-1} + \vert \lambda^M_1\vert^2 x_{k-1}^2 \\
 &= x_k \left( x_k - a x_{k-1}\right) + \vert \lambda^M_1\vert^2 x_{k-1}^2\\
 &= x_k b x_{k-2} -b x_{k-1}^2 \\
 &\simeq x_{k-1}^2 - x_k x_{k-2}
 \end{split}
 \end{equation}

 We used the $\simeq$ symbol above to denote equivalence from the Lyapunov function point of view between function $V(\cdot)$ and $aV(\cdot)$, $a > 0$. So, $V(\cdot) \simeq aV(\cdot)$
 \end{itemize} 
\end{enumerate} 
\end{proof}

It is interesting, that the expression above serves as the Lyapunov function for the very wide class of methods. Several methods, which allow two-step reformulation like in \eqref{oplyap:eq:two_step_scalar} are presented in Table \ref{oplyap:tab:two_step_methods}

\begin{table}[t]
 \centering
 \resizebox{\textwidth}{!}{
 \begin{tabular}{@{}cccc@{}}
 \toprule
 Method & $a$ & $b$ & $\rho(M)$ if $\lambda_1^M$ and $\lambda_2^M$ are conjugate pairs \\ 
 \midrule
 \begin{tabular}[c]{@{}c@{}}\textbf{HB} \cite{polyak1964some}\\ with $\alpha, \beta$\end{tabular} & $1 - \alpha\lambda + \beta$ & $-\beta$ & $\sqrt{\beta}$ \\
 \begin{tabular}[c]{@{}c@{}}\textbf{NAG}\cite{nesterov2018lectures}\\ with $\alpha, \beta$\end{tabular} & $(1 - \alpha \lambda)(1 + \beta)$ & $- (1 - \alpha \lambda)\beta$ & $\sqrt{(1 - \alpha \mu)\beta}$ \\
 \begin{tabular}[c]{@{}c@{}}\textbf{TMM}\cite{van2017fastest}\\ with $\alpha, \beta,\gamma$\end{tabular} & $(1+ \beta - \alpha (1 + \gamma) \lambda)$ & $\left(\alpha \gamma \lambda - \beta \right)$ & $\sqrt{\beta - \alpha\gamma\mu}$ \\
 \begin{tabular}[c]{@{}c@{}}\textbf{NAG-GS}\cite{leplat2022nag}\\ with $\alpha, \beta$\end{tabular} & $2\beta + (1-\beta)^2 - \alpha(1-\beta) \lambda$ & $- \beta^2$ & $\beta$ \\
 \bottomrule
 \end{tabular}
 }
 \caption{Reformulation of first-order methods with general parameters in the format given in theorem \ref{oplyap:th:scalar_lyapunov}}
 \label{oplyap:tab:two_step_scalar}
\end{table}

Therefore, we can study the hyperparameters of methods, presented in Table \ref{oplyap:tab:two_step_methods} for the meeting requirements of Theorem \eqref{oplyap:th:scalar_lyapunov}. Studying the specific requirements on the hyperparameters $\alpha, \beta, \gamma$ is of great interest and is the question of further research. 
Block matrix formulation for the vector version of the methods from Table \ref{oplyap:tab:two_step_scalar} is presented in Appendix \ref{oplyap:sec:block_matrix_formulation}.

\subsection{General $d$-dimensional case}

\begin{theorem}
 \label{oplyap:th:main_theorem}
 For the quadratic optimization problem in the form of \eqref{problem}:
 \begin{align}
 \label{oplyap:eq:full_problem}
 \min\limits_{x \in \mathbb{R}^d} f(x) = \min\limits_{x \in \mathbb{R}^d} \dfrac{1}{2} x^T W x - b^T x + c, \text{ where }W \in \mathbb{S}^d_{++}
 \end{align}
 with a unique solution $x^\star = W^{-1}b$, given any optimization method, which converges to $x^\star$ and could be written in the following form
 \[
 x_{k+1} = A x_k + B x_{k-1},
 \]
 where $A, B \in \mathbb{R}^{d \times d}$ are diagonal matrices, or, equivalently:
 \[
 z_{k+1} = M z_k, \quad M = \begin{bmatrix} 
 A & B \\
 I & 0_{d}
 \end{bmatrix} \quad 
 z_k = \begin{bmatrix} 
 x_{k} \\
 x_{k-1}
 \end{bmatrix}
 \]
 where the eigenvalues of the iteration matrix for each dimension (see the corresponding rearrangement on Figure~\ref{oplyap:fig:rearrangement}) $M_i$ forms the conjugate pairs, i.e. $(\lambda_1^{M_i})^* = \lambda_2^{M_i} = \lambda^{M_i} \; \forall i \in 1,\ldots, d$ has the following Lyapunov function:
 \begin{equation}
 \label{oplyap:eq:lyapunov_function}
 V(x_k, x_{k-1}, x_{k-2}) = \Vert x_{k-1} - x^*\Vert ^2 - \langle x_k - x^*, x_{k-2} - x^* \rangle
 \end{equation}
\end{theorem}
\begin{proof}
 \begin{enumerate}
 \item It is enough to observe, that we can easily change variables with the help of eigendecomposition $W = Q \Lambda Q^\star$ similarly as it was done in \eqref{change_var} 
 \[
 \hat{x}_{k} = Q^*(x_k - x^\star) \; \text{or} \; x_k = Q \hat{x_k} + x^\star
 \]

 Thus, our function became quadratic form with a diagonal matrix with the minimum at $\hat{x} = 0$:

 \[
 f(\hat{x}) = \frac12 \langle \hat{x}, \Lambda \hat{x}\rangle - \frac12\langle b, A^{-1}b\rangle + c
 \]

 \item Due to the possible rearrangement of matrix (Figure~\ref{oplyap:fig:rearrangement} and \eqref{oplyap:eq:rearrangement}) and, therefore, independent coordinate-wise dynamics with matrix $M_i$.
 \begin{equation*}
 M_i = \begin{bmatrix} 
 a_i & b_i \\
 1 & 0
 \end{bmatrix},
 \end{equation*}
 where $a_i, b_i$ are the $i$-th diagonal elements of matrices $A$ and $B$ we can write down the Lyapunov function for each dimension of the vector $\hat{x}$. It follows from the Theorem \ref{oplyap:th:scalar_lyapunov} that for each dimension of the vector $\hat{x} \in \mathbb{R}^d$ one can write the Lyapunov function, which will decrease monotonically if $\rho(M_i)$ and $(\lambda_1^{M_i})^* = \lambda_2^{M_i}$:

 \begin{equation*}
 V^i(\hat{x}^i_k, \hat{x}^i_{k-1}, \hat{x}^i_{k-2}) = \left(\hat{x}^i_{k-1}\right)^2 - \hat{x}^i_k \hat{x}^i_{k-2}
 \end{equation*}

 \item Now we can sum all the Lyapunov functions over dimensions $\forall i \in 1, \ldots, d$:
 
 \begin{equation*}
 \begin{split}
 V(\hat{x}_k, \hat{x}_{k-1}, \hat{x}_{k-2}) &= \sum_{i=1}^d V^i(\hat{x}^i_k, \hat{x}^i_{k-1}, \hat{x}^i_{k-2}) \\
 &= \sum_{i=1}^d \left( \left(\hat{x}^i_{k-1}\right)^2 - \hat{x}^i_k \hat{x}^i_{k-2}\right) \\
 &= \Vert \hat{x}_{k-1}\Vert ^2 -(\hat{x}_k)^T(\hat{x}_{k-2})
 \end{split}
 \end{equation*}

 \item Switching back to the original variables with $ \hat{x}_{k} = Q^*(x_k - x^\star)$
 
 \begin{equation*}
 \begin{split}
 V(x_k, x_{k-1}, x_{k-2}) &= \Vert Q^*(x_{k-1} - x^\star)\Vert ^2 -(Q^*(x_k - x^\star))^T(Q(x_{k-2} - x^\star)) \\
 &= (x_{k-1} - x^\star)^T QQ^* (x_{k-1} - x^\star) - ((x_k - x^\star))^TQQ^*((x_{k-2} - x^\star)) \\
 &= \Vert x_{k-1} - x^\star\Vert ^2 - \langle x_k - x^\star, x_{k-2} - x^\star\rangle
 \end{split}
 \end{equation*}
 \end{enumerate}
\end{proof}

\begin{figure}[h!]
 \centering
 \includegraphics[width=\linewidth]{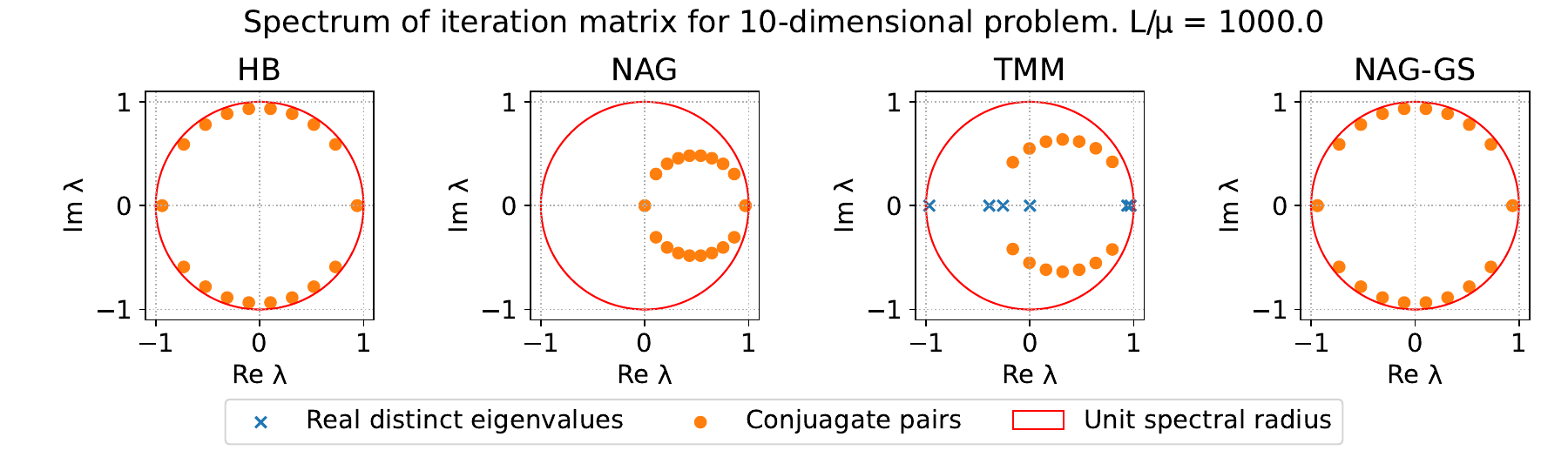}
 \includegraphics[width=\linewidth]{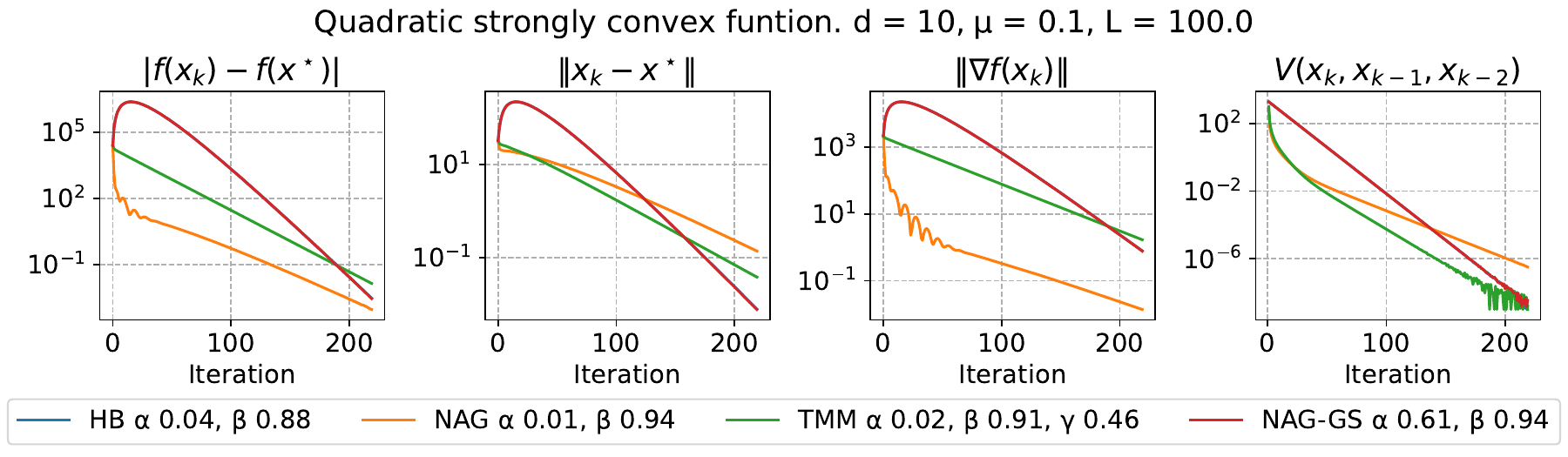}
 \caption{The correspondance between the Spectrum of iteration matrix for \textbf{HB}, \textbf{NAG}, \textbf{TMM}, \textbf{NAG-GS} methods with optimal hyperparameters applied to strongly convex $10$-dimensional quadratics and convergence characteristics.}
 \label{oplyap:fig:eigenvalues_distribution}
\end{figure}

It is important to mention, that for $d$-dimensional case the proposed Lyapunov function is a sum of geometric progressions for each coordinate with rates $\vert \lambda^{M_1} \vert, \vert \lambda^{M_2} \vert, \ldots, \vert \lambda^{M_d} \vert$ and thus the asymptotic convergence rate is determined by the worst among them, which means, that starting from some iteration number the convergence rate will be the spectral radius of the iteration matrix $\rho(M) = \max\limits_{i=1,\ldots, d} \vert \lambda^{M_i}\vert$. It is also interesting that in the \textbf{NAG} case, the eigenvalues of the iteration matrix in the multidimensional case differ significantly in magnitude, which leads to the fact that at the beginning the convergence of the Lyapunov function is determined by the convergence along those coordinates with the smallest magnitude eigenvalues (orange line in Figure~\ref{oplyap:fig:eigenvalues_distribution} from the bottom), and then it comes to the asymptotic convergence rate determined by the largest magnitude eigenvalues.

\begin{figure}[h!]
    \centering
    \includegraphics[width=\linewidth]{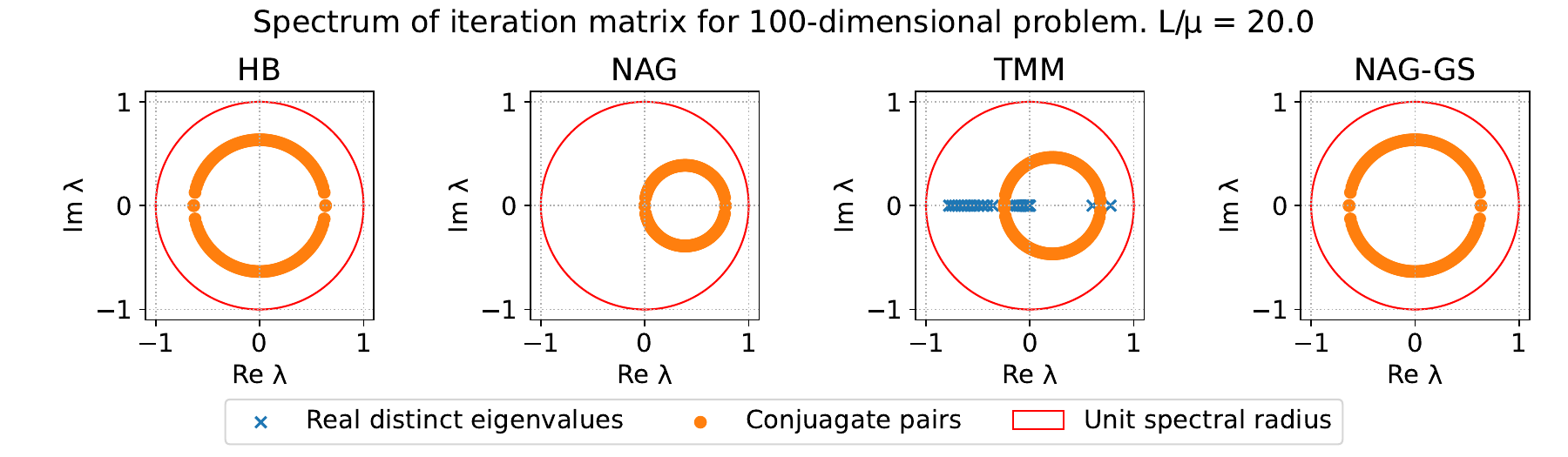}
    \includegraphics[width=\linewidth]{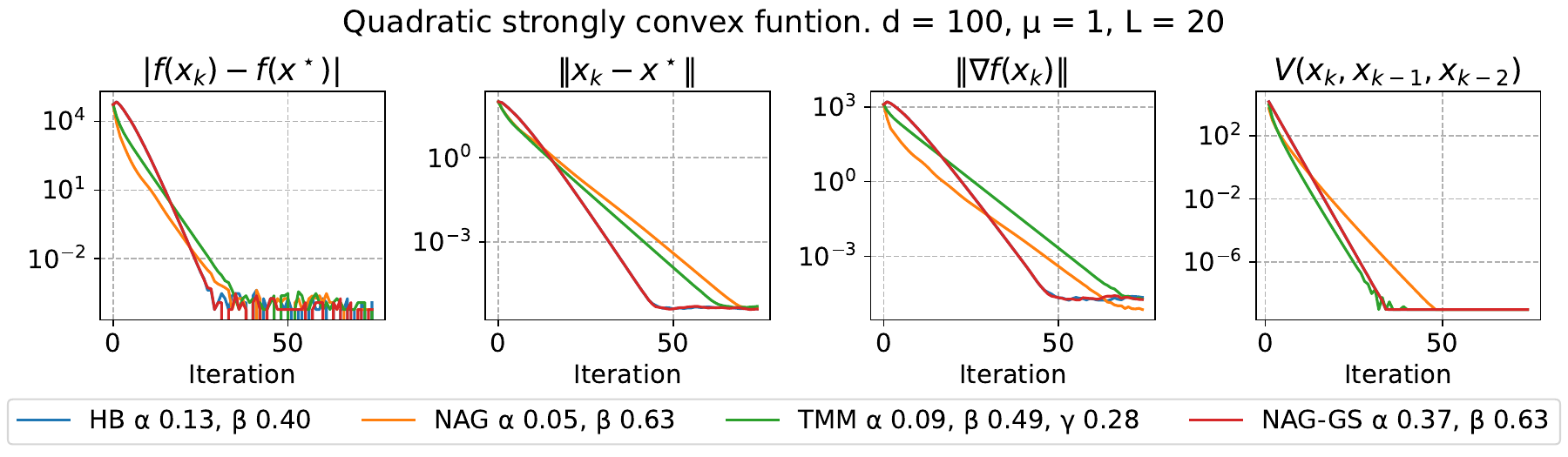}
    \caption{Dynamics of methods from Table \ref{oplyap:tab:two_step_methods} with optimal hyperparameters $\alpha^\star, \beta^\star, \gamma^\star$ applied to the strongly convex quadratic problem \eqref{oplyap:eq:full_problem}}
   \end{figure}

The spectrum of the considered method may say a lot about the convergence. For example on Figure~\ref{oplyap:fig:eigenvalues_distribution} it is easy to verify, that the \textbf{TMM} method does not satisfy the Theorem\ref{oplyap:th:main_theorem} requirements, therefore we can see, that the expression \eqref{oplyap:eq:lyapunov_function} is not Lyapunov function for the method. Moreover, we can see, that eigenvalue distribution for \textbf{HB} and \textbf{NAG-GS} significantly differs from the \textbf{NAG} and \textbf{TMM}. For the first group, the absolute values of the eigenvalues are the same (they form a circle on the complex plane), while for the latter the magnitudes of the eigenvalues vary. This is the reason why the corresponding $V(x_k, x_{k-1}, x_{k-2})$ dynamics is faster at the beginning of the optimization process and slows down at the end - convergence rate at the end depends only on the spectral radius (largest magnitude).

Note that the considered class of methods with the diagonal matrices $A$ and $B$ contains many popular methods (see Table~\ref{oplyap:sec:block_matrix_formulation}). However, the whole idea of constructing the described Lyapunov function relies essentially on the fact that we can consider the dynamics of each component of the vector x independently (see Section~\ref{oplyap:sec:reduction_to_a_scalar}). Arbitrary methods with an arbitrary iteration matrix, in general, cannot fail to be suitable for such a procedure of Lyapunov function construction.

The proposed Lyapunov function works well for a variety of scenarios. However, it is not a Lyapunov function for a general (strongly) convex optimization case. The counter-examples are provided in the corresponding sections below.

\section{Numerical experiments}

All the code for experiments is available on the GitHub Repository: \\ \href{https://github.com/MerkulovDaniil/SimpleLyapunov}{https://github.com/MerkulovDaniil/SimpleLyapunov}

\subsection{Quadratic problem}
To validate the theoretical claims, we conducted experiments on quadratic problems defined in equation \eqref{oplyap:eq:full_problem}. We applied various first-order methods, including Heavy Ball and Nesterov Accelerated Gradient, to minimize the quadratic function.

We started by randomly generating matrices $W \in \mathbb{S}^d_{++}$ of dimensions $d=100,200,500$. The generated matrices were ensured to be positive definite. Moreover, the spectrum of the matrices is uniformly spread from $\mu$ to $L$. Then, we generated random vector $x^\star \in \mathbb{R}^d$ and a vector $b \in \mathbb{R}^d$ was also calculated as $b = Ax^\star$ for each test case. For each setup, we performed iterations using various algorithms and monitored the value of the proposed Lyapunov function $ V(x_k, x_{k-1}, x_{k-2}) = \Vert x_{k-1} - x^*\Vert^2 - \langle x_k - x^*, x_{k-2} - x^* \rangle $. In our experiments, the tolerance of $V$ measuring is $10^{-9}$, which is why we can see a plateau at this level at the end.

\begin{figure}[h!]
    \centering
    \includegraphics[width=\linewidth]{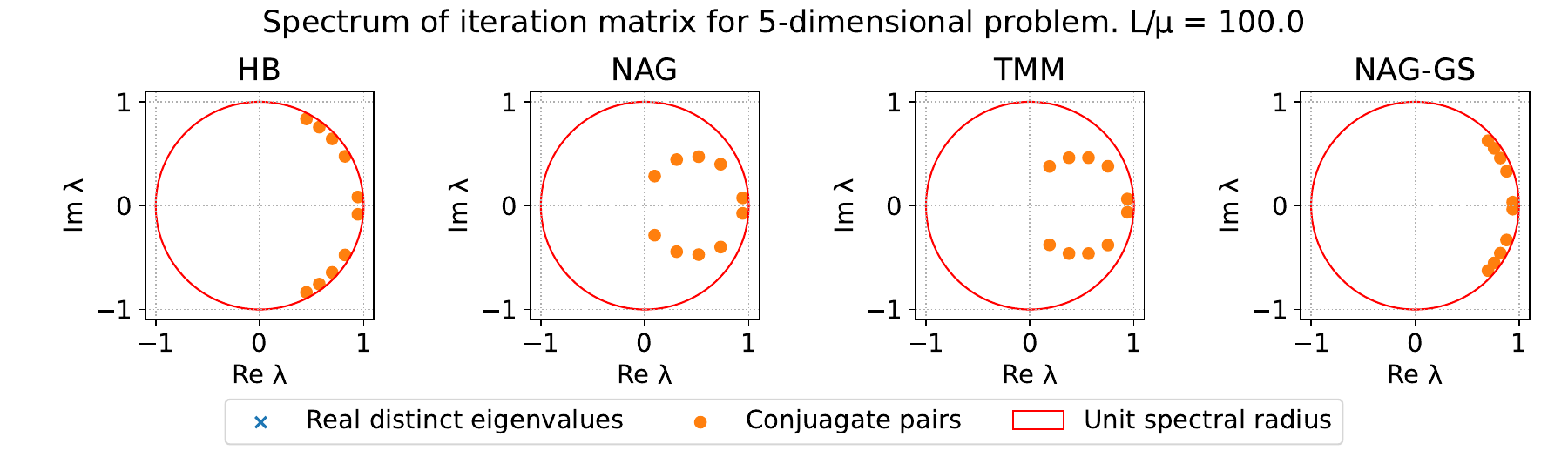}
    \includegraphics[width=\linewidth]{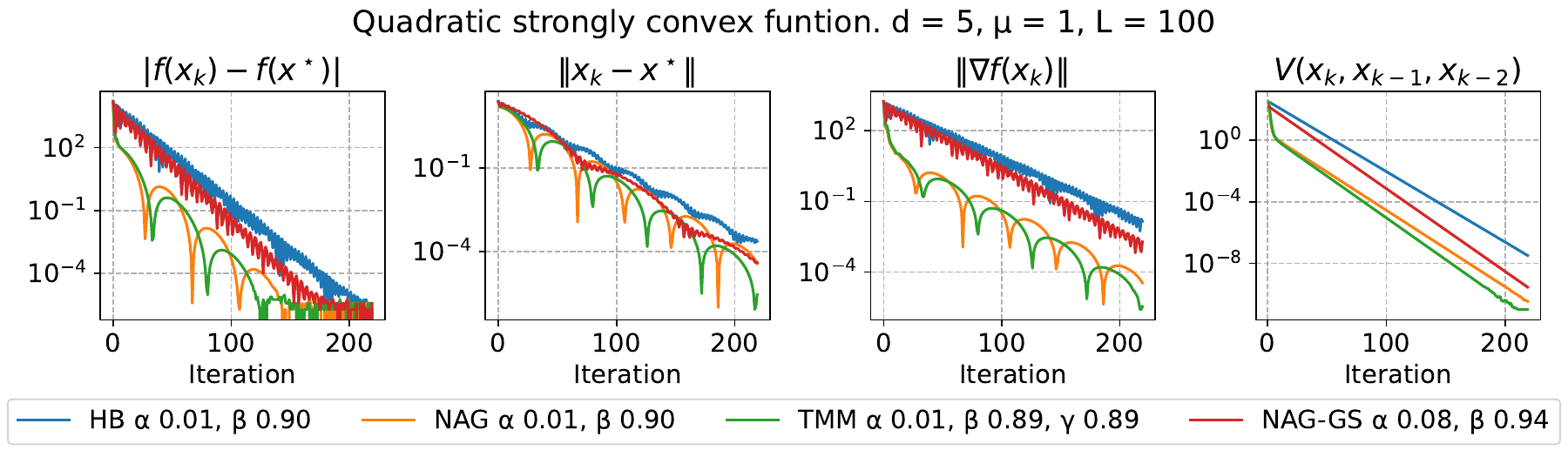}
    \caption{Dynamics of methods from Table \ref{oplyap:tab:two_step_methods} with non-optimal hyperparameters $\alpha, \beta, \gamma$ applied to the strongly convex quadratic problem \eqref{oplyap:eq:full_problem}}
    \label{oplyap:fig:non_optimal}
\end{figure}

\subsubsection{Optimal hyperparameters for methods}

The results are consistent with the theoretical predictions. The Lyapunov function monotonically decreased and approached zero as the methods converged. This indicates that our Lyapunov function provides an accurate measure of algorithmic behavior for quadratic problems. Note, that for \textbf{TMM} method we don't have theoretical guarantees for $V$ to be a Lyapunov function. Here are the results for the ill-conditioned quadratic problem:

\subsubsection{Non-optimal, but suitable hyperparameters}

It is especially interesting to look at Figure \ref{oplyap:fig:non_optimal}, where all considered methods meet Theorem \ref{oplyap:th:main_theorem} requirements, despite having non-optimal hyperparameters. Nowadays, such formulation of methods, where hyperparameters are to be tuned, is widely spread in Applications - Neural Networks training.

\subsubsection{Convex quadratic problem with $\mu=0$.}

It follows from the structure of the matrix $M$, that if the spectrum of the original matrix $W$ has $k$ zero eigenvalues, then we will have $k$ real unit eigenvalues, which corresponding summands $V^i(x_k, x_{k-1}, x_{k-2})$ will not decrease during the iteration process. Practically speaking it means, that some terms of expression \eqref{oplyap:eq:lyapunov_function} will linearly decrease, which leads to an almost linear decrease of the $V(x_k, x_{k-1}, x_{k-2})$ until some level, after that we will have some oscillations. This is supported by Figure~\ref{oplyap:fig:convex_quad}.

\begin{figure}[h!]
 \centering
 \includegraphics[width=\linewidth]{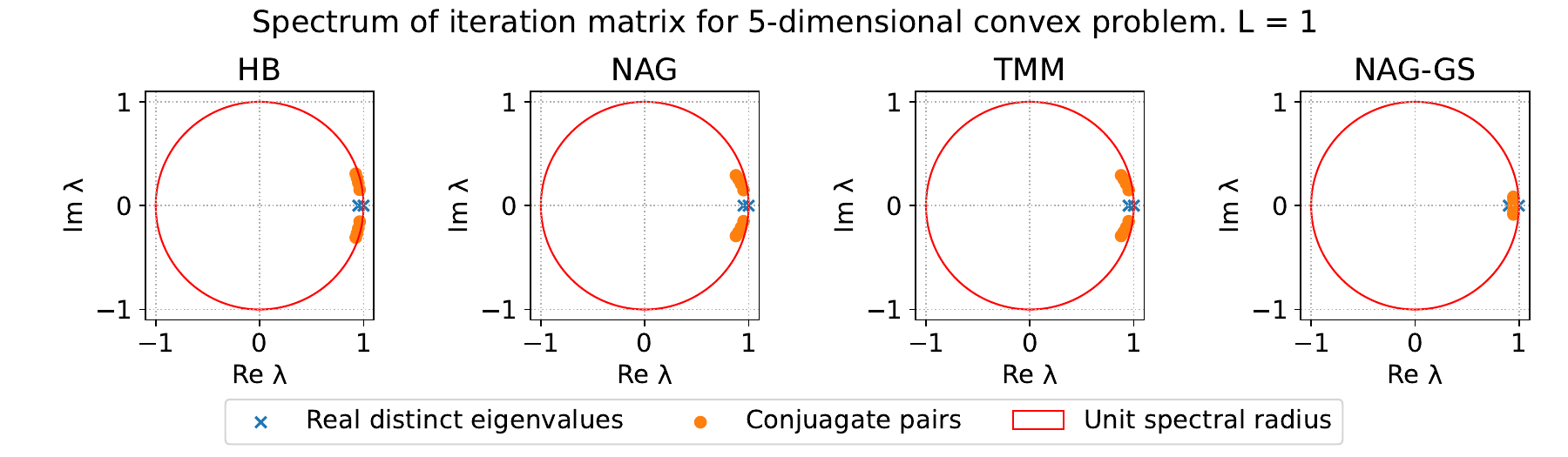}
 \includegraphics[width=\linewidth]{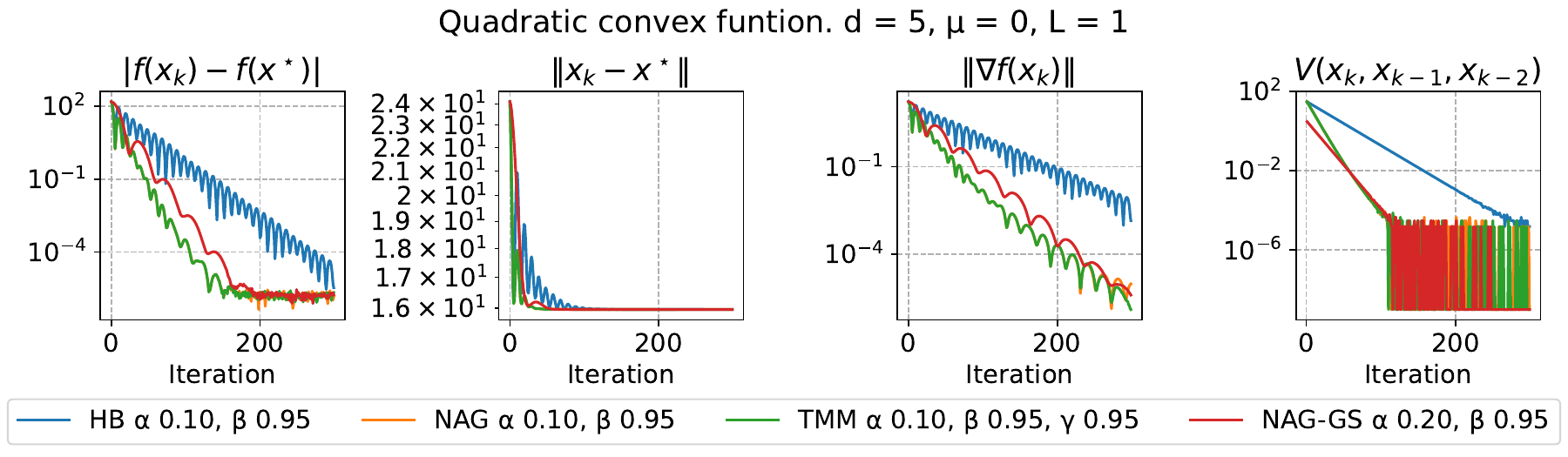}
 \caption{Dynamics of methods from Table \ref{oplyap:tab:two_step_methods} with some hyperparameters $\alpha, \beta, \gamma$ applied to the convex quadratic problem \eqref{oplyap:eq:full_problem}}
 \label{oplyap:fig:convex_quad}
\end{figure}

\subsection{Strongly convex non-quadratic problem}
We considered an example of the convex problem, where \textbf{HB} method failed to converge with optimal hyperparameters for the strongly convex function \cite{polyak1987introduction}
$$
f(x) = x^2 + \dfrac{1.99}{400}\cos(20x)
$$
It has $\mu = 0.01, L=3.99$

\begin{figure}[h!]
 \centering
 \includegraphics[width=\linewidth]{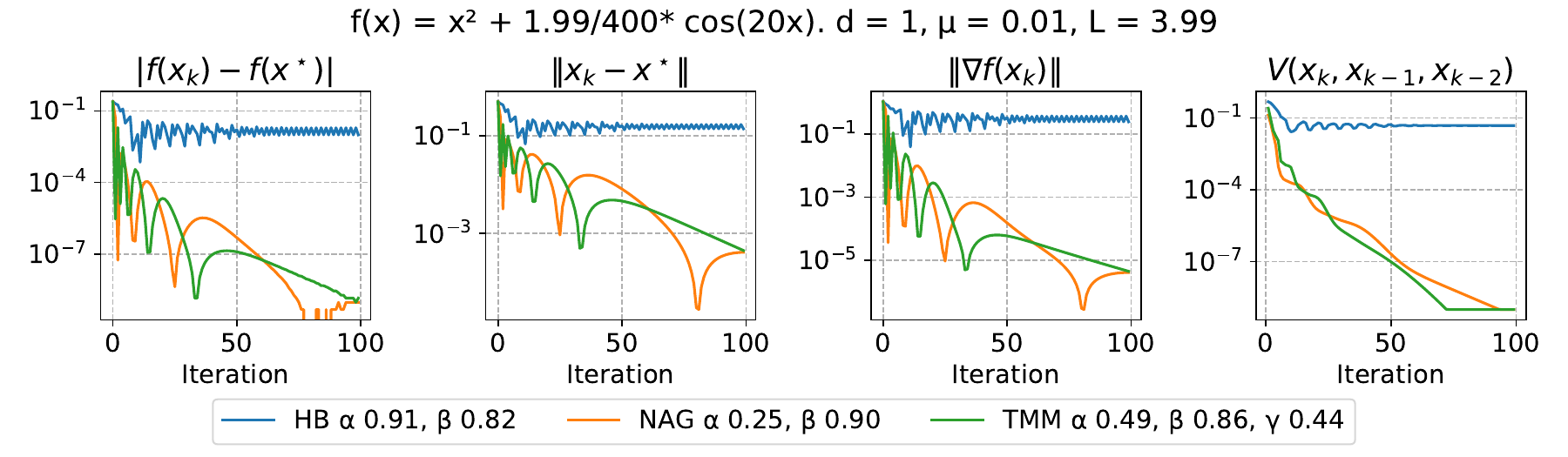}
\end{figure}

\begin{figure}[h!]
 \centering
 \includegraphics[width=\linewidth]{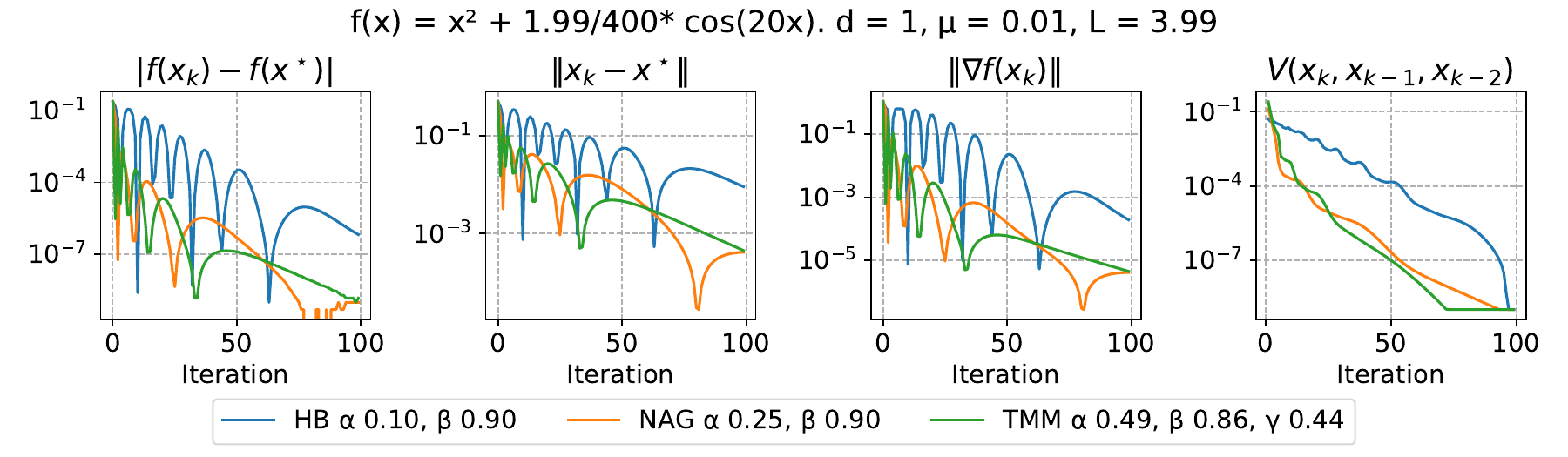}
\end{figure}

One can conclude, that such a simple function won't serve as a general-purpose Lyapunov function.

\section{Conclusion}
We presented a novel (to the best of our knowledge) approach to construct a Lyapunov function for a quadratic optimization problem and first-order algorithms, based on the bounding of the last diagonal element of the iteration matrix after Schur decomposition. It is interesting to mention, that such a simple expression serves as a Lyapunov function for the wide family of two-step methods, such as Heavy Ball, Nesterov Accelerated Gradient, Triple Momentum Method, and Nesterov Accelerated Gradient with Gauss-Seidel splitting method under some conditions, which is formulated as a main result of the paper. We have conducted experiments on quadratics, that support our claims and presented a counter-example of a general strongly convex function, where the constructed function is not a Lyapunov function.

\newpage
\printbibliography

@article{polyak1964some,
  title={Some methods of speeding up the convergence of iteration methods},
  author={Polyak, Boris T},
  journal={Ussr computational mathematics and mathematical physics},
  volume={4},
  number={5},
  pages={1--17},
  year={1964},
  publisher={Elsevier}
}

@inproceedings{ghadimi2015global,
  title={Global convergence of the heavy-ball method for convex optimization},
  author={Ghadimi, Euhanna and Feyzmahdavian, Hamid Reza and Johansson, Mikael},
  booktitle={2015 European control conference (ECC)},
  pages={310--315},
  year={2015},
  organization={IEEE}
}

@article{hammarling1982numerical,
  title={Numerical solution of the stable, non-negative definite lyapunov equation},
  author={Hammarling, Sven J},
  journal={IMA Journal of Numerical Analysis},
  volume={2},
  number={3},
  pages={303--323},
  year={1982},
  publisher={Oxford University Press}
}

@article{goh2017why,
  author = {Goh, Gabriel},
  title = {Why Momentum Really Works},
  journal = {Distill},
  year = {2017},
  url = {http://distill.pub/2017/momentum},
  doi = {10.23915/distill.00006}
}

@inproceedings{giselsson2014monotonicity,
  title={Monotonicity and restart in fast gradient methods},
  author={Giselsson, Pontus and Boyd, Stephen},
  booktitle={53rd IEEE Conference on Decision and Control},
  pages={5058--5063},
  year={2014},
  organization={IEEE}
}

@book{nesterov2003introductory,
  title={Introductory lectures on convex optimization: A basic course},
  author={Nesterov, Yurii},
  volume={87},
  year={2003},
  publisher={Springer Science \& Business Media}
}

@article{van2017fastest,
  title={The fastest known globally convergent first-order method for minimizing strongly convex functions},
  author={Van Scoy, Bryan and Freeman, Randy A and Lynch, Kevin M},
  journal={IEEE Control Systems Letters},
  volume={2},
  number={1},
  pages={49--54},
  year={2017},
  publisher={IEEE}
}

@inproceedings{nesterov1983method,
  title={A method of solving a convex programming problem with convergence rate $\mathcal{O}\left( 1/k^2\right)$},
  author={Nesterov, Yurii},
  booktitle={Doklady Akademii Nauk},
  volume={269},
  number={3},
  pages={543--547},
  year={1983},
  organization={Russian Academy of Sciences}
}

@book{nesterov2018lectures,
  title={Lectures on convex optimization},
  author={Nesterov, Yurii and others},
  volume={137},
  publisher={Springer}
}

@article{leplat2022nag,
  title={NAG-GS: semi-implicit, accelerated and robust stochastic optimizer.},
  author={Leplat, Valentin and Merkulov, Daniil and Katrutsa, Aleksandr and Bershatsky, Daniel and Oseledets, Ivan},
  year={2022}
}

@article{polyak1987introduction,
  title={Introduction to optimization. Optimization software},
  author={Polyak, Boris T},
  journal={Inc., Publications Division, New York},
  volume={1},
  pages={32},
  year={1987}
}

@book{vorontsova2021convex,
  title={Выпуклая оптимизация},
  author={Евгения Воронцова {,} Роланд Хильдебранд {,}  Александр Гасников {,} Фёдор Стонякин},
  publisher = {MIPT},
  volume         = {364},
  isbn           = {978-5-7417-0776-0},
  year={2021}
}

@article{d2021acceleration,
  title={Acceleration methods},
  author={d’Aspremont, Alexandre and Scieur, Damien and Taylor, Adrien and others},
  journal={Foundations and Trends in Optimization},
  volume={5},
  number={1-2},
  pages={1--245},
  year={2021},
  publisher={Now Publishers, Inc.}
}

@incollection{gorbunov2022recent,
  title={Recent theoretical advances in decentralized distributed convex optimization},
  author={Gorbunov, Eduard and Rogozin, Alexander and Beznosikov, Aleksandr and Dvinskikh, Darina and Gasnikov, Alexander},
  booktitle={High-Dimensional Optimization and Probability: With a View Towards Data Science},
  pages={253--325},
  year={2022},
  publisher={Springer}
}

@inproceedings{taylor2018lyapunov,
  title={Lyapunov functions for first-order methods: Tight automated convergence guarantees},
  author={Taylor, Adrien and Van Scoy, Bryan and Lessard, Laurent},
  booktitle={International Conference on Machine Learning},
  pages={4897--4906},
  year={2018},
  organization={PMLR}
}

\newpage
\appendix
\section{Two-step notation of gradient methods for quadratic minimization}
\label{oplyap:sec:block_matrix_formulation}
\begin{table}[h!]
 \centering
 \begin{tabular}{@{}ccc@{}}
 \toprule
 Method & \multicolumn{2}{c}{Iteration} \\ \midrule
 \multirow{2}{*}{\textbf{HB} \cite{polyak1964some}} & \multicolumn{2}{c}{$x_{k+1} = x_k - \alpha\nabla f(x_k) + \beta(x_k - x_{k-1}) $} \\
 & \multicolumn{2}{c}{$x_{k+1} = ((1+ \beta)I - \alpha \Lambda)x_k - \beta x_{k-1} $} \\
 Optimal $\alpha, \beta$: & \multirow{2}{*}[-12pt]{$A = (1+ \beta)I - \alpha \Lambda$} & \multirow{2}{*}[-12pt]{$B = -\beta I$} \\
 $\alpha^\star = \frac{4}{\left( \sqrt{L} + \sqrt{\mu}\right)^2}, \; \beta^\star = \frac{\left( \sqrt{L} - \sqrt{\mu}\right)^2}{\left( \sqrt{L} + \sqrt{\mu}\right)^2}$ & & \\ \midrule
 
 \multirow{2}{*}{\textbf{NAG} \cite{nesterov2003introductory}} & \multicolumn{2}{c}{$\begin{cases}y_{k+1} = x_k + \beta (x_k - x_{k-1}) \\ x_{k+1} = y_{k+1} - \alpha \nabla f(y_{k+1}) \end{cases}$} \\
 & \multicolumn{2}{c}{$x_{k+1} = (1+\beta)\left(I - \alpha \Lambda\right)x_k - \beta\left(I - \alpha \Lambda\right)x_{k-1}$} \\
 Optimal $\alpha, \beta$: & \multirow{2}{*}[-7pt]{$A = (1+\beta)\left(I - \alpha \Lambda\right)$} & \multirow{2}{*}[-7pt]{$B = - \beta\left(I - \alpha \Lambda\right)$} \\
 $\alpha^\star = \frac{1}{L}, \; \beta^\star = \frac{\sqrt{L} - \sqrt{\mu}}{\sqrt{L} + \sqrt{\mu}}$ & & \\ \midrule
 
 \multirow{2}{*}{\textbf{TMM} \cite{van2017fastest}} & \multicolumn{2}{c}{$x_{k+1} = (1+\beta)x_k - \beta x_{k-1} -\alpha \nabla f\left( (1 + \gamma)x_k - \gamma x_{k-1}\right)$} \\
 & \multicolumn{2}{c}{$x_{k+1} = ((1+ \beta)I - \alpha (1 + \gamma) \Lambda)x_k + \left(\alpha \gamma \Lambda - \beta I \right)x_{k-1} $} \\
 Optimal $\alpha, \beta, \gamma^\star$: $\rho = 1 - \sqrt{\frac{\mu}{L}}$ & \multirow{2}{*}[-7pt]{$A = (1+ \beta)I - \alpha (1 + \gamma) \Lambda$} & \multirow{2}{*}[-7pt]{$B =\alpha \gamma \Lambda - \beta I $} \\
 $\alpha^\star = \frac{1 + \rho}{L}, \; \beta^\star = \frac{\rho^2}{2 - \rho}, \; \gamma^\star = \frac{\rho^2}{(1 + \rho)(2-\rho)}$ & & \\ \midrule
 
 \multirow{2}{*}{\textbf{NAG-GS} \cite{leplat2022nag}} & \multicolumn{2}{c}{$\begin{cases}y_{k} = \beta y_{k-1} + (1-\beta)x_k - \alpha \nabla f(x_k)\\ x_{k+1} = \beta x_{k} + (1-\beta)y_k \end{cases}$} \\
 & \multicolumn{2}{c}{$x_{k+1} = (2 \beta + (1-\beta)^2)I - \alpha (1-\beta) \Lambda x_k - \beta^2 x_{k-1} $}\\
 Optimal $\alpha, \beta$: & \multirow{2}{*}[-7pt]{$A = (2 \beta + (1-\beta)^2)I - \alpha (1-\beta) \Lambda $} & \multirow{2}{*}[-7pt]{$B = - \beta^2 I$} \\
 $\alpha^\star = \frac{2 + 2\sqrt{\frac{L}{\mu}}}{L + \mu + 2\sqrt{\mu L}}, \; \beta^\star = \frac{L - \mu}{L + \mu + 2 \sqrt{\mu L}}$ & & \\ \bottomrule
 \end{tabular}
 \caption{Correspondence between several accelerated methods for strongly convex functions and its reformulations concerning two-step notation.}
 \label{oplyap:tab:methods}
\end{table}

\section{Experiments}

\subsection{Quadratic Problem}

$$
\min\limits_{x \in \mathbb{R}^d} \dfrac{1}{2} x^T W x - b^T x + c, \text{ where }W \in \mathbb{S}^d_{++}
$$

\begin{figure}[h!]
 \centering
 \includegraphics[width=\linewidth]{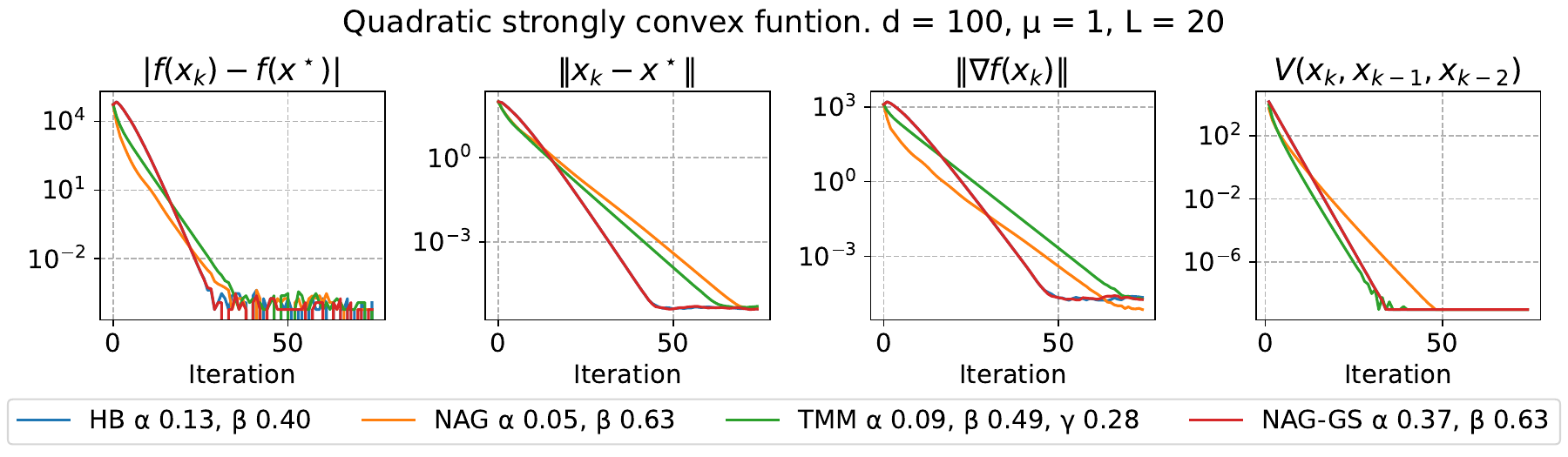}
\end{figure}

\begin{figure}[h!]
 \centering
 \includegraphics[width=\linewidth]{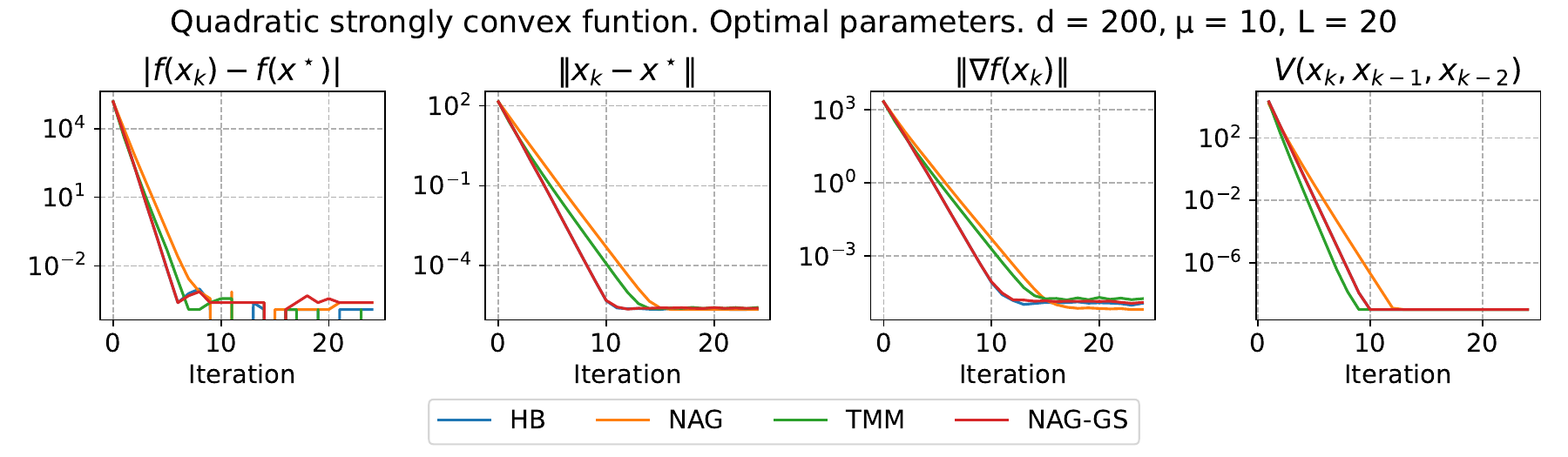}
\end{figure}

\begin{figure}[h!]
 \centering
 \includegraphics[width=\linewidth]{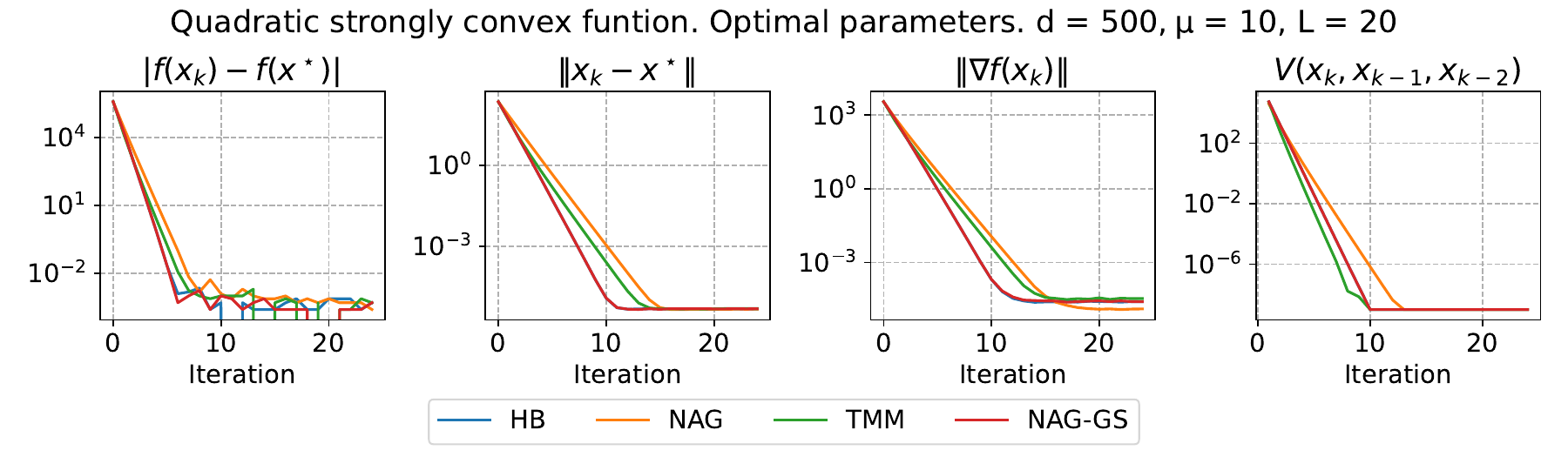}
\end{figure}

\begin{figure}[h!]
 \centering
 \includegraphics[width=\linewidth]{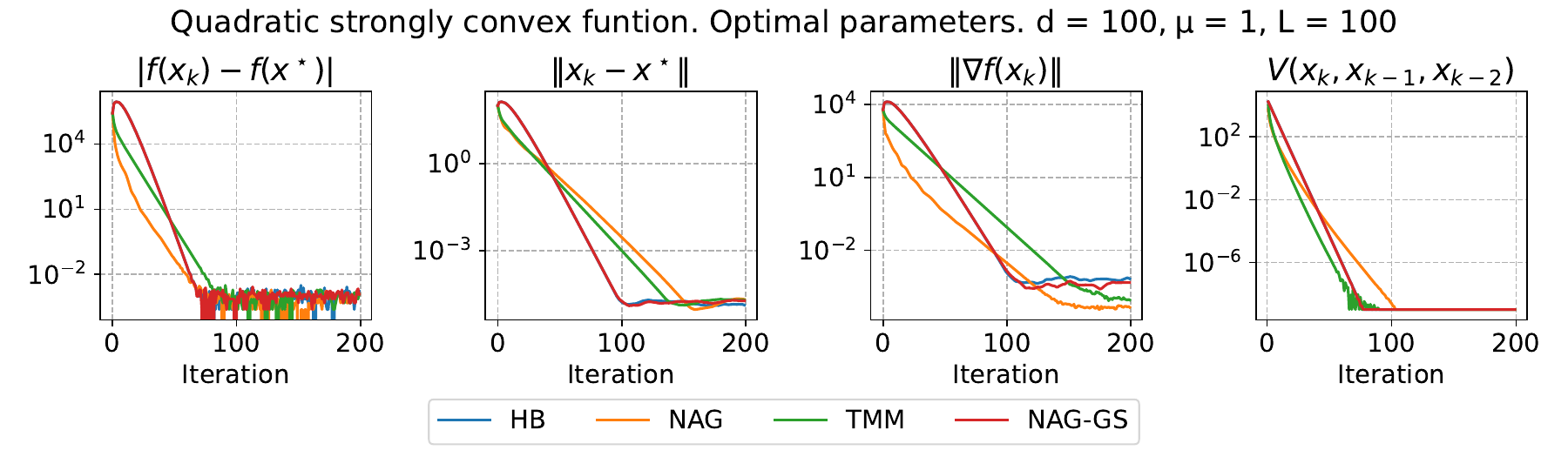}
\end{figure}

\begin{figure}[h!]
 \centering
 \includegraphics[width=\linewidth]{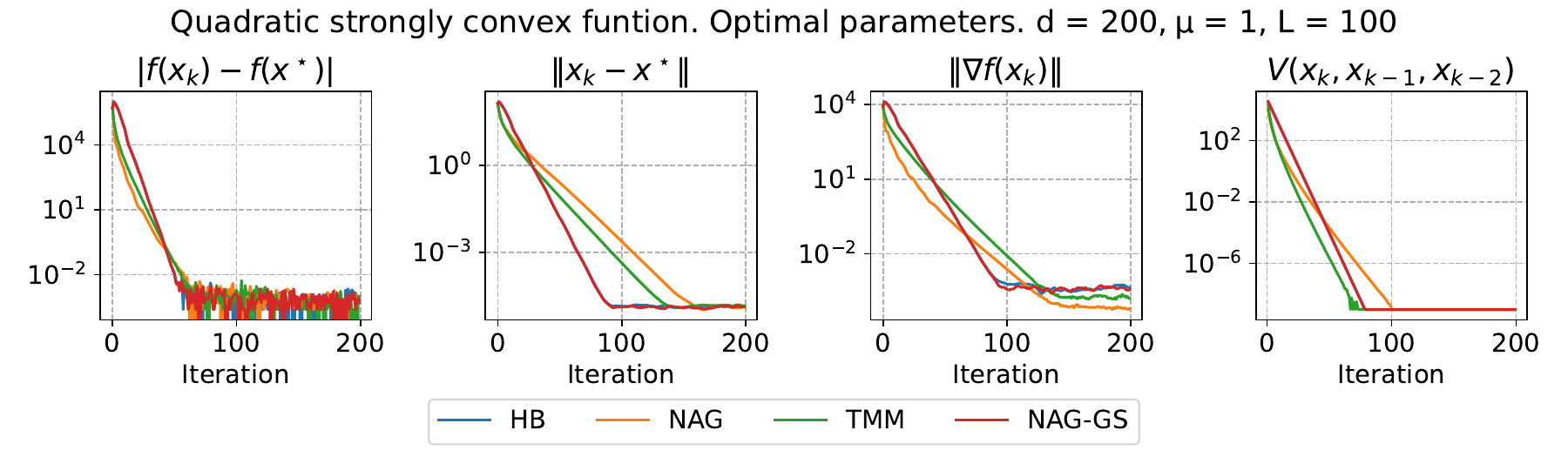}
\end{figure}

\begin{figure}[h!]
 \centering
 \includegraphics[width=\linewidth]{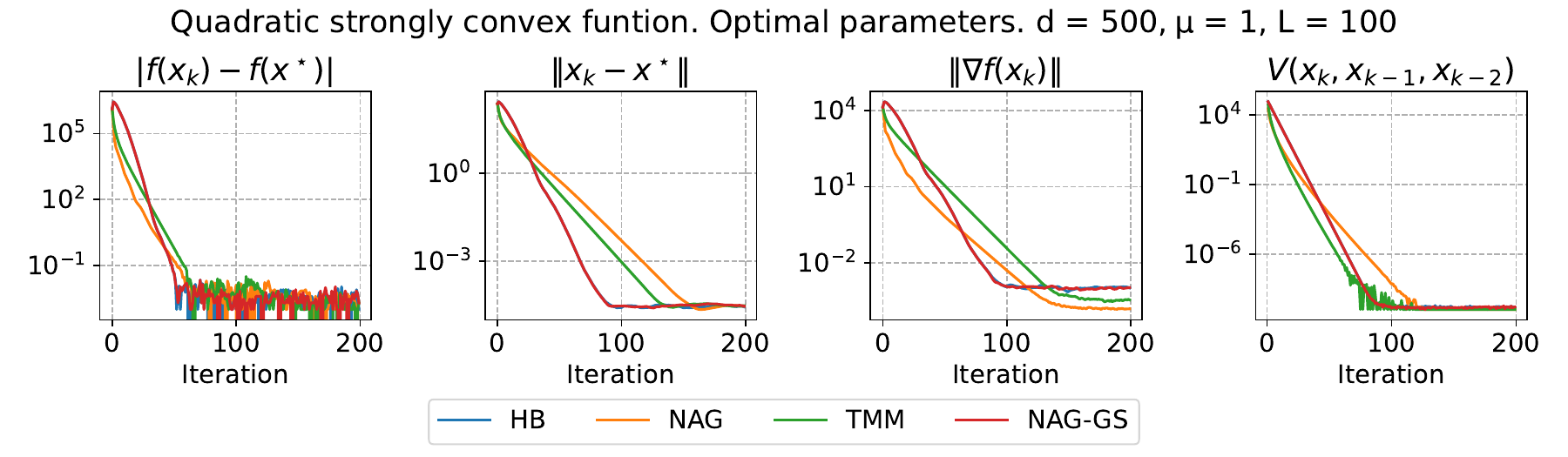}
\end{figure}

\subsection{Convex problem}
$$
\min\limits_{x \in \mathbb{R}^d} f(x) = e^{\Vert x\Vert ^2_2}
$$

\begin{figure}[h!]
 \centering
 \includegraphics[width=\linewidth]{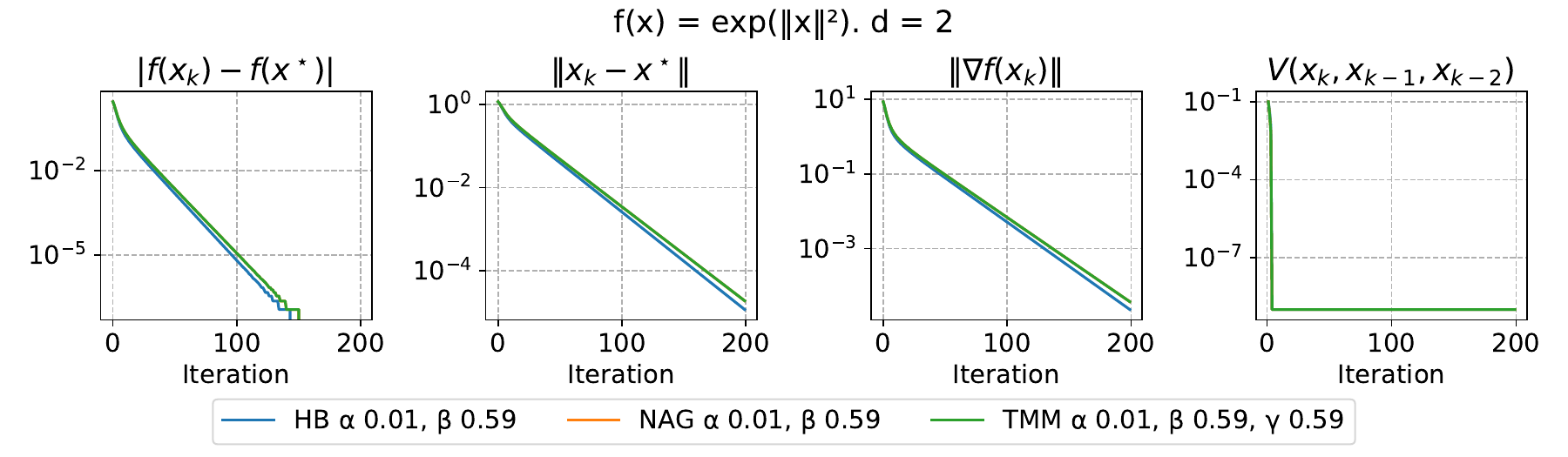}
\end{figure}

\subsection{Non-convex problem}

$$
\min\limits_{x,y \in \mathbb{R}^2} f(x, y) = (1 - x)^2 + 100(y - x^2)^2, \; x^\star = (1, 1)
$$

\begin{figure}[h!]
 \centering
 \includegraphics[width=\linewidth]{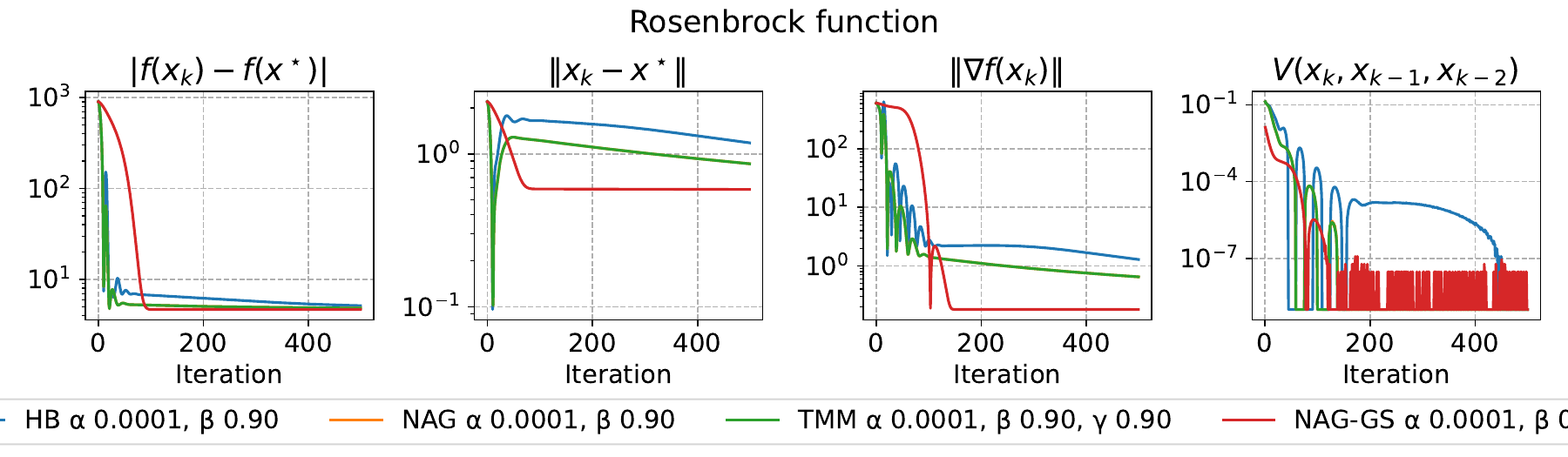}
\end{figure}

\end{document}